\documentclass[12pt]{amsart}
\usepackage{amsmath,amssymb,amsxtra,amscd}
\usepackage[margin=3cm]{geometry} 
\usepackage[hypertex]{hyperref}
\usepackage[all]{xy}

 \numberwithin{equation}{section}
\newtheorem{theorem}{Theorem}[section]
\newtheorem{lemma}[theorem]{Lemma} 
\newtheorem{proposition}[theorem]{Proposition} 
\newtheorem{corollary}[theorem]{Corollary} 
\theoremstyle{definition}
\newtheorem{definition}[theorem]{Definition} 
\newtheorem{notation}[theorem]{Notation} 
\newtheorem{remark}[theorem]{Remark}

\newcommand{\C}{\mathbb{C}}
\newcommand{\Z}{\mathbb{Z}} 
\newcommand{\N}{\mathbb{N}}
\newcommand{\Q}{\mathbb{Q}} 
\newcommand{\R}{\mathbb{R}}
\newcommand{\bS}{\mathbb{S}} 
\newcommand{\tbS}{\widetilde{\bS}} 
\newcommand{\PP}{\mathbb{P}} 
\newcommand{\LL}{\mathbb{L}} 
\newcommand{\hLL}{\widehat{\LL}}
\newcommand{\M}{\mathbb{M}}

\newcommand{\cO}{\mathcal{O}} 
\newcommand{\cS}{\mathcal{S}} 
 
\newcommand{\cE}{\mathcal{E}}

\newcommand{\cD}{\mathcal{D}} 
\newcommand{\cG}{\mathcal{G}}
\newcommand{\cI}{\mathcal{I}} 
\newcommand{\cM}{\mathcal{M}}

\newcommand{\hT}{{\widehat{T}}} 
\newcommand{\htau}{\hat{\tau}}
\newcommand{\hd}{\hat{d}}

\newcommand{\tW}{\widetilde{W}} 

\newcommand{\bx}{\mathbf{x}} 
\newcommand{\bN}{\mathbf{N}}
\newcommand{\bt}{\mathbf{t}}
\newcommand{\bbf}{\mathbf{f}}
\newcommand{\bV}{\mathbf{V}} 
\newcommand{\by}{\mathbf{y}} 
\newcommand{\ybase}{\mathbf{y}^*} 
\newcommand{\bD}{\mathbf{D}}
\newcommand{\bepsilon}{\boldsymbol{\epsilon}}
\newcommand{\tbepsilon}{\tilde{\bepsilon}} 

\newcommand{\frs}{\mathfrak{s}} 
\newcommand{\frI}{\mathfrak{I}} 
 
\newcommand{\frX}{\mathfrak{X}} 
\newcommand{\frf}{\mathfrak{f}}

\newcommand{\vir}{{\rm vir}} 

\renewcommand{\sec}{{\rm sec}}  
\newcommand{\loc}{{\rm loc}} 
\newcommand{\eff}{{\rm eff}}
\newcommand{\noneq}{{\rm noneq}}
\newcommand{\nilp}{{\rm nilp}}

\newcommand{\Spec}{\operatorname{Spec}} 
\newcommand{\Spf}{\operatorname{Spf}}
 
\newcommand{\id}{\operatorname{id}} 
\newcommand{\ev}{\operatorname{ev}} 
\newcommand{\pt}{\operatorname{pt}} 
\newcommand{\Frac}{\operatorname{Frac}} 
\newcommand{\Lie}{\operatorname{Lie}} 
\newcommand{\Eff}{\operatorname{Eff}} 
 
\newcommand{\Hom}{\operatorname{Hom}} 
 
\newcommand{\End}{\operatorname{End}} 
\newcommand{\rank}{\operatorname{rank}}

\newcommand{\GM}{\operatorname{GM}}
\newcommand{\Gr}{\operatorname{Gr}}
\newcommand{\hotimes}{\mathbin{\widehat\otimes}}

\def\corr#1{\left\langle#1 \right\rangle} 
\def\parfrac#1#2{\frac{\partial #1}{\partial #2}} 

\begin{document} 

\title[A mirror construction for toric quantum cohomology]
{A mirror construction for the big equivariant quantum cohomology of toric manifolds}
\author{Hiroshi Iritani} 
% \thanks{Department of Mathematics, Graduate 
% School of Science, Kyoto University, Kitashirakawa-Oiwake-cho, 
% Sakyo-ku, Kyoto, 606-8502, Japan, 
% Tel: +81 (0)75 753 2670, 
% E-mail: iritani@math.kyoto-u.ac.jp}
% \subjclass[2010]{14N35, 53D45, 14J33, 53D37}
\email{iritani@math.kyoto-u.ac.jp} 
\address{Department of Mathematics, Graduate School of Science, 
Kyoto University, Kitashirakawa-Oiwake-cho, Sakyo-ku, 
Kyoto, 606-8502, Japan}
%\date{\today}
\begin{abstract} 
We identify a certain universal Landau-Ginzburg model as a 
mirror of the big equivariant quantum cohomology of a 
(not necessarily compact or semipositive) toric manifold. 
The mirror map and the primitive form 
are constructed via Seidel elements and 
shift operators for equivariant quantum cohomology. 
Primitive forms in non-equivariant theory 
are identified up to automorphisms of the mirror. 
\end{abstract} 

\maketitle

\section{Introduction} 
The big quantum cohomology of a smooth projective variety $X$ 
is a formal family of commutative products $\star_\tau$ 
on the space $H^*(X)[\![Q]\!]$ parameterized by 
$\tau\in H^*(X)$,  
which gives the cup product at the limit $Q=\tau=0$. 
Here $Q$ is the Novikov variable that keeps track of 
the degree of rational curves in $X$. 
It is well-known that the quantum product $\star_\tau$ 
defines a flat connection on the tangent bundle of $H^*(X)$ 
with parameter $z\in \C^\times$ 
\[
\nabla_\alpha = \partial_\alpha + z^{-1} (\alpha \star_\tau) 
\qquad \text{for $\alpha \in H^*(X)$}  
\]
called the \emph{quantum connection}. 
This connection defines the structure of a $\cD$-module on 
the tangent sheaf of $H^*(X)$; by a slight abuse of language we refer to 
this $\cD$-module also as quantum connection. 
In this paper, we study mirror symmetry for the quantum connection   
of a toric variety. More precisely, we identify the quantum connection 
with the Gauss-Manin connection associated with a function (Landau-Ginzburg potential) 
on the algebraic torus $(\C^\times)^D$ of dimension $D = \dim X$.

Givental \cite{Givental:ICM, Givental:toric_mirrorthm} 
and Hori-Vafa \cite{Hori-Vafa} proposed that 
a mirror of a toric variety is a Laurent polynomial 
function on $(\C^\times)^D$ of the following form: 
\begin{equation} 
\label{eq:undeformed_potential} 
F(x) = Q^{\beta_1} x^{b_1} + \cdots + Q^{\beta_m} x^{b_m},  
\qquad x\in (\C^\times)^D
\end{equation} 
where $b_1,\dots,b_m\in \bN \cong \Z^D$ are primitive generators 
of one-dimensional cones of the fan $\Sigma$ of the toric variety $X_\Sigma$,  
$\beta_i = \beta(b_i) \in H_2(X_\Sigma,\Z)$ is a certain curve class 
determined by the choice of a splitting of the fan sequence (see \S \ref{subsec:GM}).  
Givental's mirror theorem \cite{Givental:toric_mirrorthm} implies, 
when the toric variety $X_\Sigma$ is compact and $c_1(X_\Sigma)$ is semipositive, 
that the small\footnote
{Small means that the parameter $\tau$ is restricted to lie in $H^2$; 
big means that the parameter space is the whole cohomology group.} 
quantum cohomology of $X_\Sigma$ is isomorphic to the Jacobi ring of $F(x)$ 
and that the small quantum connection of $X_\Sigma$ is isomorphic to 
the twisted de Rham cohomology 
$H^D(\Omega_{(\C^\times)^D}^\bullet[z], z d + dF \wedge )$ 
equipped with the Gauss-Manin connection. 
Givental \cite{Givental:toric_mirrorthm} also showed that 
a mirror of \emph{equivariant} quantum cohomology is given by adding 
a logarithmic term to the potential: 
$F_\lambda(x) = F(x) + \sum_{i=1}^D \lambda_i \log x_i$, 
where $\lambda_1,\dots,\lambda_D$ are the equivariant parameters 
for a torus $T \cong (\C^\times)^D$ acting on the toric variety $X_\Sigma$. 

A generalization of this result to big quantum cohomology has been 
studied by Barannikov \cite{Barannikov:projective} and 
Douai-Sabbah \cite{Douai-Sabbah:II}. 
They obtained big quantum cohomology mirrors of (weighted) projective spaces 
by adding to $F(x)$ monomial terms which form a basis of the Jacobi ring. 
This leads to an isomorphism of Frobenius manifolds between 
the A-model (quantum cohomology) and the B-model (singularity theory). 
% There are also studies on non-compact/local case 
% \cite{Givental:elliptic, Stienstra,Klemm}, 
% non-semipositive case 
% \cite{Iritani:genmir,Iritani:coLef}, toric orbifolds \cite{Iritani:integral, CCIT:mirrorthm}, 
% a relation to the GKZ system 
% \cite{Stienstra, Iritani:integral, Reichelt-Sevenheck:logFrob, 
% Reichelt-Sevenheck:nonaffine, Konishi-Minabe, Mochizuki}, 
% tropical geometry \cite{Gross:P2} and Lagrangian Floer theory 
% \cite{FOOO, CLLT, Gonzalez-Iritani:potential}. 

The aim of this paper is to describe mirror symmetry for 
both \emph{big} and \emph{equivariant} quantum cohomology. 
It turns out that the mirror has a simple description. 
We introduce a \emph{universal Landau-Ginzburg potential} $F_\lambda(x;\by)$  
of the form: 
\[
F_\lambda(x;\by) = \sum_{k } y_k Q^{\beta(k)} x^k - \lambda \cdot \log x. 
\]
Here the sum is taken over \emph{all} lattice points $k\in \bN$ 
in the support $|\Sigma|$ of the fan, $\beta(k)\in H_2(X_\Sigma,\Z)$ 
is a certain curve class (it is determined by the conditions that 
$\beta(b_i) = \beta_i$, $1\le i\le m$ and that $\beta(k+l) = \beta(k) + \beta(l)$ 
whenever $k, l$ belong to the same cone of $\Sigma$; see \S \ref{subsec:GM}) and  
$\by = \{y_k\}$ is an infinite set of parameters. 
The infinitely many parameters $\by$ of the Landau-Ginzburg 
potential $F_\lambda(x;\by)$ correspond, under mirror symmetry, 
to the parameter $\tau$ in the equivariant cohomology $H_T^*(X_\Sigma)$ 
which is also infinite dimensional.  
The twisted de Rham cohomology of the universal potential $F_\lambda$ 
gives the \emph{Gauss-Manin system} $\GM(F_\lambda)$ 
which forms a $\cD$-module over the $\by$-space 
via the Gauss-Manin connection. 
Roughly speaking, $\GM(F_\lambda)$ consists of volume forms 
$\varphi(x;\by) \frac{dx_1}{x_1} \cdots \frac{dx_D}{x_D}$ on the torus 
$(\C^\times)^D$ and the $\cD$-module structure is defined 
in such a way that the oscillatory integrals 
\[
\int \varphi(x;\by) e^{F_\lambda(x;\by)/z} \frac{dx_1}{x_1} \cdots \frac{dx_D}{x_D} 
\]
are solutions. 
We refer to \S\ref{subsec:GM} for a precise definition; we define 
$\GM(F_\lambda)$ over the ring of formal power series 
in the parameters $\by$. 
Our main theorem is stated as follows. 

\begin{theorem}[see Theorem \ref{thm:mirror_isom}, 
Corollary \ref{cor:equiv_Jacobi_ring} for precise statements]
\label{thm:main_introd} 
Let $\bN\cong \Z^D$ be a lattice and let $\Sigma$ be a fan 
in $\bN \otimes \R$ which defines a smooth semi-projective\footnote
{This is equivalent to $X_\Sigma$ being a GIT quotient of a vector space. 
We do not assume that $X_\Sigma$ is projective or $c_1(X_\Sigma)$ is semipositive.} 
toric variety $X_\Sigma$ having a torus-fixed point. 
There is a formal invertible change of variables (mirror map) between 
the A-model parameter $\tau \in H^*_T(X_\Sigma)$ 
and the B-model parameter $\by$ such that 
the Gauss-Manin system $\GM(F_\lambda)$ of $F_\lambda$ 
is identified with the big equivariant quantum connection of $X_\Sigma$ 
and that the Jacobi ring of $F_\lambda$ is identified with the equivariant 
quantum cohomology of $X_\Sigma$, under the mirror map. 
\end{theorem} 
The key ingredients of the proof are Seidel elements and shift operators for equivariant 
quantum cohomology. In fact, 
the above theorem follows almost as a \emph{formal consequence} 
of properties of these operators.   
A Seidel element is an invertible element of quantum cohomology 
associated to a Hamiltonian circle action on a symplectic manifold, 
introduced by Seidel \cite{Seidel:pi1}. 
This can be ``lifted'' to the equivariant setting and yields a 
shift operator for equivariant parameters 
\cite{Okounkov-Pandharipande:Hilbert, BMO:Springer, 
Maulik-Okounkov, Iritani:shift}. 
The mirror map and the mirror isomorphism in the 
above theorem can be described as follows: 

\begin{theorem}[a generalization of \cite{Gonzalez-Iritani:Selecta};  
Proposition \ref{prop:mirror_map}] 
The mirror map $\by \mapsto \tau(\by)$ is characterized by 
the differential equation 
\[
\parfrac{\tau(\by)}{y_k} = S_k(\tau(\by)) 
\]
together with a certain asymptotic initial condition,  
where $S_k(\tau)$ is the Seidel element associated to the 
$\C^\times$-action $k \in \bN$. Here we identify $\bN$ with 
the cocharacter lattice $\Hom(T,\C^\times)$ of the natural 
torus $T$ acting on $X_\Sigma$. 
\end{theorem} 

\begin{theorem}[Theorem \ref{thm:y+}, Corollary \ref{cor:primitive_Upsilon}]
Introduce another infinite set $\by^+ = \{y_{k,n} : k\in \bN\cap |\Sigma|, 
n =1,2,3,\dots\}$ of variables 
and consider a formal family of elements in $\GM(F_\lambda)$:  
\[
\omega(\by^+) = \exp\left(\sum_{k\in \bN\cap|\Sigma|} \sum_{n=1}^\infty 
y_{k,n} z^{n-1} Q^{\beta(k)} x^k\right) 
\frac{dx_1}{x_1}\wedge \cdots \wedge \frac{dx_D}{x_D}. 
\]
The image $\Upsilon(\by,\by^+,z)
\in H^*_T(X_\Sigma)[z][\![Q]\!][\![\by,\by^+]\!]$ 
of $\omega(\by^+)$ under the mirror isomorphism 
in Theorem \ref{thm:main_introd} is characterized by the differential equation: 
\begin{align*} 
\parfrac{\Upsilon(\by,\by^+,z)}{y_{k,n}} & = [z^{n-1} \bS_k(\tau(\by))]_+ 
\Upsilon(\by,\by^+,z) \qquad n=0,1,2,\dots, 
% \parfrac{\Upsilon(\by,\by^+,z)}{y_{k,n}} & = z^{n-1} 
% \bS_k(\tau(\by)) \Upsilon(\by,\by^+,z) 
% \qquad (n\ge 1)  
\end{align*}
together with a certain asymptotic initial condition, 
where $\bS_k(\tau)$ denotes the shift operator associated to 
the $\C^\times$-action $k\in \bN$ 
and we set $y_{k,0} :=y_k$. 
In particular, a primitive form in the sense of K.~Saito \cite{SaitoK:primitiveform} 
is given by $\omega(\by^+)$ with $\by^+$ 
satisfying $\Upsilon(\by,\by^+,z)=1$. 
\end{theorem} 

Mirror symmetry for non-equivariant big quantum cohomology 
follows immediately by taking a non-equivariant limit of 
Theorem \ref{thm:main_introd}. 
In order to obtain a Landau-Ginzburg potential and 
a (cochain-level) primitive form in the non-equivariant setting, we need 
to choose a formal map 
$(\frs,\frf)\colon H^*(X_\Sigma)[\![Q]\!] \to H_T^*(X_\Sigma)[\![Q]\!]  
\times H^*_T(X_\Sigma)[z][\![Q]\!]$ such that 
the non-equivariant limit of $(\frs(\sigma),\frf(\sigma))$ equals 
$(\sigma,1)$ for $\sigma \in H^*(X_\Sigma)[\![Q]\!]$. 
The pull-back $\frs^*F$ of the universal Landau-Ginzburg potential 
with $\lambda = 0$ 
gives a universal unfolding of the original potential 
\eqref{eq:undeformed_potential} (for some choice of $\frs$; 
see Proposition \ref{prop:universal_unfolding}) 
and the data $(\frs,\frf)$ associates a primitive form $\zeta_{(\frs,\frf)}$ 
of the Gauss-Manin system $\GM(\frs^*F)$. 
These primitive forms associated with various choices 
are related to each other by co-ordinate changes of the mirror.

\begin{theorem}[Theorems \ref{thm:noneq}, \ref{thm:orbit_space} 
and Corollary \ref{cor:noneq_primitive}]
The Gauss-Manin system $\GM(\frs^*F)$ of $\frs^* F$ 
is isomorphic to the non-equivariant 
big quantum connection of $X_\Sigma$. 
Moreover, oscillatory primitive forms $\exp(\frs^* F/z)\zeta_{(\frs,\frf)}$ 
associated with various data $(\frs,\frf)$ are related to each other 
by reparametrizations of the mirror. 
\end{theorem} 

We observe that reparametrizations of the mirror form an infinite-dimensional 
formal group $J\cG$.  
The group $J\cG$ reduces the equivariant theory to the non-equivariant one: 
in terms of Givental's Lagrangian cone \cite{Givental:symplectic}, 
the non-equivariant Givental cone can be regarded as 
the orbit space of the equivariant Givental cone 
under a $J\cG$-action (see Theorem \ref{thm:orbit_space} and 
Remark \ref{rem:cone_quotient}). 

The mirror map and primitive forms can be calculated concretely 
in terms of the following hypergeometric series, 
called the \emph{extended $I$-function} \cite{CCIT:mirrorthm}: 
\[
I(\by,z) = z e^{\sum_{i=1}^m u_i \log y_i/z} 
\sum_{\ell \in \hLL_\eff} \by^\ell Q^{d(\ell)}
\left( \prod_{i=1}^m 
\frac{\prod_{c=-\infty}^0 (u_i+ cz)}
{\prod_{c=-\infty}^{\ell_{b_i}}(u_i+cz)} \right) 
\frac{1}{\prod_{k \in G} \ell_k! z^{\ell_k}}. 
\]
where $\hLL_\eff$ is the index set defined in \S \ref{subsec:mirrormap} 
and $G = (\bN \cap |\Sigma|) \setminus \{b_1,\dots,b_m\}$. 
We deduce the following theorem from basic properties of shift operators, 
without relying on the mirror theorem \cite{CCIT:mirrorthm} 
for the extended $I$-function. 
\begin{theorem}[Corollary \ref{cor:primitive}] 
We set $y_k(z) = y_k + \sum_{n=1}^\infty y_{k,n} z^n$ and 
$\by(z) = \{y_k(z) : k\in \bN \cap |\Sigma|\}$. 
The primitive form of the equivariant mirror 
is given by $\omega(\by^+)$ for $\by^+$ such that 
one has $I(\by(z),z) = z( 1 + O(z^{-1}))$. 
Moreover, for such $\by^+$, the asymptotics $I(\by(z),z) = z + \tau(\by) + O(z^{-1})$ 
determines the mirror map $\tau(\by)$. 
\end{theorem} 

The primitive form in this paper is given as a formal power series in the parameters $\by$, 
and should be thought of as a ``formal primitive form'' 
in the sense of Li-Li-Saito \cite{LiLiSaito:polyvector} 
(see also \cite{SaitoM:Brieskorn_II}). 
Note that our primitive form is defined over the Novikov ring.  

Cho-Oh \cite{Cho-Oh} and Fukaya-Oh-Ohta-Ono 
\cite{FOOO:toricI,FOOO:toricII,FOOO:toric_mirrorsymmetry} 
constructed the Landau-Ginzburg potential as a generating function 
of open Gromov-Witten invariants. 
Their potentials were computed for compact semi-positive 
toric manifolds by Chan-Lau-Leung-Tseng \cite{CLLT:Seidel} 
via Seidel representation (see also \cite{Gonzalez-Iritani:potential}). 
It is natural to ask if our inverse mirror map $\tau \mapsto 
y_k(\tau)$ (and the function $\tau \mapsto \by^+(\tau)$ 
giving the primitive form) is a generating function of 
certain open Gromov-Witten invariants. 
We would like to draw attention to the related approach of 
Gross \cite{Gross:tropical_P2} using tropical disc counting and that 
of Gonz\'alez-Woodward \cite{Gonzalez-Woodward:tmmp} 
using quantum Kirwan maps. 
We also remark that the present approach via Seidel representation 
is closely related to the viewpoint of Teleman \cite{Teleman:gauge_mirror} 
on mirror symmetry. 

\begin{remark} 
\label{rem:powerseries_y}
When we talk about formal power series in $\by$ and $\by^+$, 
we always mean a function on the formal neighbourhood of the base point given by 
$y_{b_1} = \cdots = y_{b_m} = 1$, $y_k =0$ for 
$k\in G:=(\bN\cap|\Sigma|) \setminus \{b_1,\dots,b_m\}$ and 
$y_{k,n} = 0$ for all $k\in \bN\cap|\Sigma|$ and $n\ge 1$. 
This base point corresponds to the original potential 
\eqref{eq:undeformed_potential}. 
\end{remark} 

\subsection*{Acknowledgements} 
I thank Eduardo Gonz\'alez and Si Li for very helpful discussions. 
Joint work \cite{Gonzalez-Iritani:Selecta} 
with Eduardo Gonz\'alez was a starting point of the present research.  
Si Li suggested me to consider deformations 
of $F(x)$ by infinitely many monomials. 
This research is supported by JSPS Kakenhi Grant Number 
16K05127, 16H06337, 
25400069, 26610008, 23224002, 25400104.

\subsection*{Notation} 

\begin{itemize} 
\item $T\cong (\C^\times)^{\dim T}$ is an algebraic torus; 
we write $\hT := T\times \C^\times$;  
\item $X$ is a smooth $T$-variety (satisfying the assumption in 
\S \ref{sec:generalities}); 
\item $X_\Sigma$ is a smooth toric variety 
associated to a fan $\Sigma$; 
in this case $T \cong (\C^\times)^{\dim X_\Sigma}$;  
\item $(\lambda,z) \in \Lie(T) \times \Lie(\C^\times) = \Lie(\hT)$ denote 
variables on the Lie algebra; 
\item $H^*_\hT(\pt) \cong \C[\lambda,z]$ is the ring  
of polynomial functions on $\Lie(\hT)$; 
\item $H^*_\hT(X) = H^*_T(X)[z]$ where $\hT$ acts on $X$ 
via the projection $\hT \to T$; 
\item $H^*_{\hT}(X)_\loc := H^*_{\hT}(X)\otimes_{H^*_\hT(\pt)} 
\Frac(H^*_\hT(\pt)) = H^*_T(X)\otimes_{H_T^*(\pt)} 
\Frac(H^*_T(\pt)[z])$; 
\item All (co)homology groups have $\C$ coefficients 
unless otherwise specified;  
\item $\Psi(k)$ is given in Notation \ref{nota:Psi}; 
\item $y_k$ is a variable associated to a lattice point $k\in |\Sigma|$; 
$y_i = y_{b_i}$ for $1\le i\le m$. 
\end{itemize} 

\section{Shift operators in Equivariant Gromov-Witten Theory}
\label{sec:generalities} 
In this section we recall basic definitions of equivariant Gromov-Witten 
invariants and shift operators. 
Shift operators first appeared in the work of Okounkov-Pandharipande 
\cite{Okounkov-Pandharipande:Hilbert} for quantum cohomology of 
Hilbert schemes of points on $\C^2$; they are also studied  
by Braverman-Maulik-Okounkov \cite{BMO:Springer}, 
Maulik-Okounkov \cite{Maulik-Okounkov} and the author 
\cite{Iritani:shift}. 
Let $T\cong (\C^\times)^{\dim T}$ be an algebraic torus. 
Let $X$ be a smooth variety over $\C$ equipped with 
an algebraic $T$-action. We assume the following conditions. 
\begin{itemize} 
\item $X$ is semi-projective, i.e.~the natural map 
$X \to \Spec(H^0(X,\cO))$ is projective. 
\item The set of $T$-weights of $H^0(X,\cO)$ is contained in 
a strictly convex cone in $\Hom(T,\C^\times)\otimes \R$ 
and $H^0(X,\cO)^T = \C$. 
\end{itemize} 
In this paper we only need the case where 
$X$ is a toric variety, but 
the shift operator makes sense for general $X$ as above. 
The above conditions ensure that the 
$T$-fixed set $X^T$ is projective, and also that 
$X$ is equivariantly formal, i.e.~$H^*_T(X)$ is a 
free $H_T^*(\pt)$-module and one has a (non-canonical) 
isomorphism 
$H^*_T(X) \cong H^*(X) \otimes_\C H^*_T(\pt)$ 
as an $H_T^*(\pt)$-module, 
see \cite[Proposition 2.1]{Iritani:shift}. 
These conditions make equivariant Gromov-Witten invariants 
well-defined and ensure the existence of a non-equivariant limit for quantum 
cohomology. 

\subsection{Formal power series ring} 
\label{subsec:power_series}
Let $\Eff(X)\subset H_2(X,\Z)$ denote the semigroup generated by 
effective curves. For a module (or a ring) $M$, we write $M[\![Q]\!]$
for the space of formal power series of the form: 
\[
\sum_{d\in \Eff(X)} a_d Q^d, \qquad  a_d \in M.  
\]
Here $Q$ is a formal parameter called the \emph{Novikov variable}. 
For a countable set $\bx = \{x_1,x_2,x_3,\dots\}$ of variables, 
the space $M[\![\bx]\!]$ of formal power series in $\bx$ 
with coefficients in $M$ consists of formal sums of the form: 
\[
\sum_{I} a_I \bx^I, \qquad a_I \in M 
\]
where the index $I$ ranges over all sequences $(i_1,i_2,i_3,\dots)$ 
of non-negative integers such that $\sum_{n=1}^\infty i_n<\infty$ 
and we set $\bx^I = \prod_{n=1}^\infty x_n^{i_n}$. 
The space $M[\![\bx]\!]$ can be also described as 
the projective limit of the spaces $M[\![x_1,\dots,x_n]\!]$. 
Note that we shall abuse notation 
when we use the variable $\by$ and write $M[\![\by]\!]$, 
see \S\ref{subsec:mirrormap} and Remark \ref{rem:powerseries_y}. 
% The topology on these spaces is described as follows. 
% For a finite set $A \subset (\Z_{\ge 0})^{\oplus \infty}$ of indices, 
% let $M[\![\bx]\!]_A \subset M[\![\bx]\!]$ be the submodule consisting 
% of power series \eqref{eq:powerseries}

Recall that a topology on a module (or ring) is said to be linear 
if the fundamental neighbourhood system of $0$ is given by 
submodules (resp.~ ideals). 
Let $M$ be a linearly topologized module (or ring) 
and let $\{M_\nu \subset M\}$ 
denote the fundamental neighbourhood system of $0$. 
The topology on $M[\![\bx]\!]$ is defined by  
the following fundamental neighbourhood system of $0$: 
\[
%M[\![\bx]\!]_{\nu,n} = 
%M_\nu[\![\bx]\!] + \sum_{i=1}^n x_i^n M[\![\bx]\!] 
%+ \sum_{i=n+1}^\infty x_i M[\![\bx]\!] 
M[\![\bx]\!]_{\nu,\cI} = 
\left\{ \sum_{I} a_I x^I: \text{$a_I \in M_\nu$ 
for all $I \in \cI$} \right\}  
\]
where $\cI$ ranges over all finite sets of exponents $I$. 
The topology on $M[\![Q]\!]$ is defined similarly. 
The convergence in $M[\![\bx]\!]$ (or in $M[\![Q]\!]$) 
is the coefficient-wise convergence: a sequence in $M[\![\bx]\!]$ 
converges if and only if the coefficient of $x^I$ converges in $M$ 
for each $I$. 
When $M$ is complete, 
the spaces $M[\![Q]\!]$, $M[\![\bx]\!]$ are also complete. 

\subsection{Quantum cohomology and quantum connection} 
For equivariant cohomology classes 
$\alpha_1,\dots,\alpha_n \in H^*_T(X)$, 
$d\in H_2(X,\Z)$ and 
non-negative integers $c_1,\dots,c_n\in \Z_{\ge 0}$, 
we have \emph{equivariant Gromov-Witten invariants}  
\[
\corr{\alpha_1\psi^{c_1},\alpha_2 \psi^{c_2},\dots,\alpha_n\psi^{c_n}
}_{0,n,d}^{X,T} = \int_{[X_{0,n,d}]_{\vir}} 
\prod_{i=1}^n \psi_i^{c_i}\ev_i^*(\alpha_i)   
\]
taking values in the fraction ring $\Frac(H^*_T(\pt))$ of  
$H^*_T(\pt)$. Here $X_{0,n,d}$ 
denotes the moduli space of genus-zero stable maps to $X$ 
of degree $d$ and with $n$ markings, $[X_{0,n,d}]_{\vir}$ 
denotes its virtual fundamental class, $\ev_i \colon X_{0,n,d} \to X$ 
is the evaluation map at the $i$th marking and $\psi_i$ is the 
first Chern class of the universal cotangent line bundle at the $i$th marking. 
When $X$ is not proper, the right-hand side 
is defined by the virtual localization formula 
\cite{Atiyah-Bott, Graber-Pandharipande} and thus belongs to $\Frac(H^*_T(\pt))$. 

The \emph{equivariant quantum product} $\star_\tau$ with 
$\tau \in H^*_T(X)$ is given by 
\[
(\alpha\star_\tau\beta, \gamma) = \sum_{n\ge 0} 
\sum_{d\in \Eff(X)} \frac{Q^d}{n!} 
\corr{\alpha,\beta,\gamma,\tau,\dots,\tau}_{0,n+3,d}^{X,T} 
\]
for $\alpha,\beta,\gamma\in H^*_T(X)$, 
where $(\alpha,\beta) = \int_X \alpha\cup\beta$ is the equivariant 
Poincar\'{e} pairing taking values in $\Frac(H^*_T(\pt))$. 
Let $T_0,\dots,T_N$ be a basis of $H^*_T(X)$ 
over $H^*_T(\pt)$ and write $\tau = \sum_{i=0}^N \tau^i T_i$. 
The product $\star_\tau$ defines a commutative ring structure on 
\[
H_T^*(X)[\![Q]\!][\![\tau]\!] := H_T^*(X)[\![Q]\!][\![\tau^0,\dots,\tau^N]\!]. 
\]
Notice that the product $\alpha\star_\tau \beta$ lies in 
$H^*_T(X)[\![Q]\!][\![\tau^0,\dots,\tau^N]\!]$, i.e.~we do not 
need to invert equivariant parameters. 
This follows from our assumption that $X$ is semi-projective 
(see \cite[\S 2.3]{Iritani:shift}). 

The \emph{quantum connection} $\nabla$ is a pencil of flat connections 
on the tangent bundle $T H_T^*(X) = H_T^*(X) \times H_T^*(X)$ 
of $H_T^*(X)$ defined by 
\[
\nabla_\alpha = \partial_\alpha + z^{-1} (\alpha\star_\tau)  
\]
where $z$ is the pencil parameter, $\tau\in H^*_T(X)$ denotes 
a point on the base, $\alpha\in H^*_T(X)$ 
and $\partial_\alpha$ denotes the directional derivative. 
This is known to be flat, and admits a fundamental solution 
$M(\tau,z)\in \End(H^*_T(X))[\![z^{-1}]\!][\![Q]\!][\![\tau]\!]$ 
such that 
\[
\partial_\alpha \circ M(\tau,z) = M(\tau,z) \circ \nabla_\alpha. 
\]
In this paper, we use the following fundamental solution  
\cite[\S 1]{Givental:elliptic}, 
\cite[Proposition 2]{Pandharipande:afterGivental}:  
\begin{equation} 
\label{eq:fundsol} 
(M(\tau,z) \alpha,\beta) = 
(\alpha,\beta) + 
\sum_{\substack{d\in \Eff(X), n\ge 0 \\ (d,n) \neq (0,0)}} 
\frac{Q^d}{n!} 
\corr{\alpha, \tau,\dots,\tau, \frac{\beta}{z-\psi}}_{0,n+2,d}^{X,T} 
\end{equation} 
where $1/(z-\psi)$ should be expanded in power series 
$\sum_{n=0}^\infty \psi^n z^{-n-1}$. 
We regard the pencil parameter $z$ as an equivariant 
parameter for an additional $\C^\times$. 
Set $\hT = T \times \C^\times$ 
and consider the $\hT$-action on $X$ induced by the natural 
projection $\hT \to T$.  
By the localization method, we find that $M(\tau,z)$ defines 
an operator
\[
M(\tau,z) \colon H^*_\hT(X) [\![Q]\!][\![\tau]\!] \to 
H^*_\hT(X)_\loc[\![Q]\!] [\![\tau]\!]. 
\] 
where $H^*_\hT(X)_\loc= H^*_\hT(X) \otimes_{H_\hT^*(\pt)} 
\Frac(H_\hT^*(\pt))$ is called the \emph{Givental space}. 

\subsection{Shift operators}
\label{subsec:shift} 
For a cocharacter $k\colon \C^\times \to T$ of $T$, 
we say that $k$ is \emph{semi-negative}  
if $k$ pairs with every $T$-weight of $H^0(X,\cO)$ 
non-positively. Here we adopt the convention that $T$ acts on a function 
$f \in H^0(X,\cO)$ as $(t \cdot f)(x) := f(t^{-1} x)$. 
We consider a shift operator associated 
to a semi-negative cocharacter. 

For a cocharacter $k$ of $T$, we consider the space 
\begin{equation} 
\label{eq:associated_bundle} 
E_k := \left(X \times \left( \C^2\setminus \{(0,0)\}\right)\right) 
\big/\C^\times 
\end{equation} 
where $\C^\times$ acts on $X \times \C^2$ by 
$s \cdot (x, (v_1,v_2)) = (s^k x, ( s^{-1}v_1, s^{-1}v_2))$ 
and $s^k$ denotes the image of $s\in \C^\times$ under $k$. 
The space $E_k$ is a fiber bundle over $\PP^1$ with fiber $X$. 
The group $\hT = T\times \C^\times$ acts on $E_k$ by 
$(t,u) \cdot [x, (v_1,v_2)] = [tx, (v_1,uv_2)]$. 
Let $X_0$ denote the fiber of $E_k$ at $[1,0]\in \PP^1$ 
and $X_\infty$ denote the fiber of $E_k$ at $[0,1]\in \PP^1$. 
Note that the induced $\hT$-actions on $X_0$ and $X_\infty$ are given by 
\begin{align*} 
(t,u) \cdot x &= t \cdot x  && \text{for $x\in X_0$;} \\ 
(t,u) \cdot x & =t u^k \cdot x  && \text{for $x\in X_\infty$.}
\end{align*} 
We have an isomorphism $\Phi_k \colon H^*_\hT(X_0) \cong 
H_\hT^*(X_\infty)$ induced by the identity map $X_0\cong X_\infty$ 
and the group automorphism $\hT \to \hT$, $(t,u) \mapsto (t u^k, u)$. 
Notice that $\Phi_k$ is not a homomorphism of $H_\hT^*(\pt)$-modules  
but satisfies the property: 
\[
\Phi_k(f(\lambda,z) \alpha) = f(\lambda + k z, z) \Phi_k(\alpha) 
\]
for $f(\lambda,z) \in H^*_\hT(\pt)$. Here $\lambda \in \Lie(T)$ 
and $z\in \Lie(\C^\times)$ are equivariant parameters for 
$T$ and $\C^\times$ respectively and we identify $H^*_\hT(\pt)$ 
with the ring of polynomial functions on $\Lie(T)\times \Lie(\C^\times)$. 
We have an isomorphism \cite[Lemma 3.7]{Iritani:shift} 
\[
H^*_\hT(E_k) \cong \left\{(\alpha,\beta) \in H^*_\hT(X_0) \oplus
H^*_\hT(X_\infty) : \alpha  - \Phi_k^{-1} \beta \equiv 0 \mod z \right\} 
\]
which sends $\gamma$ to $(\gamma|_{X_0},\gamma|_{X_\infty})$. 
We write $\htau \in H^*_\hT(E_k)$ for the lift of $\tau \in H^*_T(X)$  
such that $\htau|_{X_0} = \tau$ and $\htau|_{X_\infty} = \Phi_k(\tau)$. 
The assignment $\tau \mapsto \htau$ is not $H^*_T(\pt)$-linear. 

Let $H_2^\sec(E_k,\Z)$ denote the subset of $H_2(E_k,\Z)$ 
consisting of section classes $\hd \in H_2(E_k,\Z)$ such that 
$\pi_*(\hd)=[\PP^1]$, where 
$\pi \colon E_k \to \PP^1$ is the natural projection. 
We set $\Eff(E_k)^\sec := \Eff(E_k) \cap H_2^\sec(E_k,\Z)$. 
Consider the $\C^\times$-action on $X$ induced by 
$k\colon \C^\times \to T$ and the $T$-action on $X$. 
To each $\C^\times$-fixed point $x\in X$, we can associate a 
section $\sigma_x = \{x\}\times \PP^1$ of $\pi \colon E_k\to \PP^1$. 
When $k$ is semi-negative, 
there exists a unique connected component $F_{\min}$ of the 
$\C^\times$-fixed set $X^{\C^\times}$ such that $\C^\times$-action  
has only positive weights on the normal bundle of $F_{\min}$ 
in $X$ (see \cite[\S 3.2]{Iritani:shift}). 
The section class associated to a fixed point in $F_{\min}$
is called the \emph{minimal section class} 
and is denoted by $\sigma_{\min}(k)$.

\begin{lemma}[{\cite[Lemma 3.5, Lemma 3.6]{Iritani:shift}}] 
Let $k$ be a semi-negative cocharacter of $T$. Then 
$E_k$ is semi-projective and 
$\Eff(E_k)^\sec = \sigma_{\min}(k) + \Eff(X)$.  
\end{lemma} 

\begin{definition}[shift operator] 
Let $k\colon\C^\times \to T$ be a semi-negative cocharacter. 
For $\tau \in H^*_T(X)$, 
define a map $\tbS_k(\tau) \colon H_\hT^*(X_0)[\![Q]\!] 
\to H_\hT^*(X_\infty)[\![Q]\!]$ 
by 
\[
\left(\tbS_k(\tau) \alpha, \beta\right) 
= \sum_{n=0}^\infty \sum_{\hd\in \Eff(E_k)^\sec}  
\frac{Q^{\hd-\sigma_{\min}(k)}}{n!} 
\corr{\iota_{0*}\alpha, \htau,\dots,\htau, \iota_{\infty*}\beta
}_{0,n+2,\hd}^{E_k,\hT}  
\]
where $\alpha \in H_\hT^*(X_0)$, $\beta \in H_\hT^*(X_\infty)$ 
and $\iota_0 \colon X_0 \to E_k$, $\iota_\infty \colon X_\infty \to E_k$ 
are the natural inclusions. 
We define $\bS_k(\tau) := \Phi_k^{-1} \circ \tbS_k(\tau) 
\colon H_\hT^*(X_0)[\![Q]\!] \to H_\hT^*(X_0)[\![Q]\!]$. 
Note that $\tbS_k(\tau),\bS_k(\tau)$ are defined 
without inverting equivariant parameters, which again follows from the fact that 
$E_k$ is semi-projective, see \cite[Remark 3.10]{Iritani:shift}.  
Note also that $\tbS_k(\tau)$ is $H_\hT^*(\pt)$-linear, but 
$\bS_k(\tau)$ satisfies $\bS_k(\tau)( f(\lambda,z) \alpha) 
= f(\lambda -kz, z) \bS_k(\tau) \alpha$ for 
$f(\lambda,z) \in H^*_\hT(\pt)$ and 
$\alpha \in H^*_\hT(X)$. 
\end{definition} 

We also introduce a (constant) shift operator acting on 
the Givental space $H_\hT^*(X)_\loc$. 
\begin{definition}[shift operator on $H_\hT^*(X)_\loc$] 
\label{def:shift_Givental} 
Let $k\colon \C^\times \to T$ be a semi-negative cocharacter. 
Let $X^T = \bigsqcup_i F_i$ denote the decomposition of 
the $T$-fixed set $X^T$ into connected components. 
Let $N_i$ be the normal bundle of $F_i$ in $X$ and 
let $N_i = \bigoplus_\alpha N_{i,\alpha}$ be the decomposition 
into $T$-eigenbundles, where $T$ acts on $N_{i,\alpha}$ by 
the weight $\alpha \in \Hom(T,\C^\times)$. 
We write $c(N_{i,\alpha}) = \prod_{j=1}^{\rank(N_{i,\alpha})} 
(1 + \rho_{i,\alpha,j})$ with $\rho_{i,\alpha,j}$ being the 
virtual Chern roots of $N_{i,\alpha}$. 
Let $\sigma_i(k) \in H_2(E_k,\Z)^\sec$ 
denote the section class of $E_k$ given by a $T$-fixed point in $F_i$. 
We set 
\[
\Delta_i(k) :=Q^{\sigma_i(k) - \sigma_{\min}(k)} 
\prod_{\alpha} \prod_{j=1}^{\rank (N_{i,\alpha})} 
\frac{\prod_{c=-\infty}^0 (\rho_{i,\alpha,j}  + \alpha + cz)} 
{\prod_{c=-\infty}^{-\alpha \cdot k} (\rho_{i,\alpha,j} + \alpha + cz)}. 
\]
Using the localization isomorphism \cite{Atiyah-Bott} 
\[
\iota^* \colon H^*_\hT(X)_\loc \cong 
H^*_\hT(X^T)_\loc = 
\bigoplus_i H^*(F_i) \otimes_\C \Frac(H^*_\hT(\pt))
\]
given by the restriction to the fixed set $X^T$, we define 
the operator $\cS_k \colon H_\hT^*(X)_\loc \to H_\hT^*(X)_\loc$ by 
the commutative diagram: 
\[
\xymatrix{
H_\hT^*(X)_\loc \ar[rrr]^{\cS_k} \ar[d]_{\iota^*} 
& & & H_\hT^*(X)_\loc \ar[d]_{\iota^*} \\ 
H_\hT^*(X^T)_\loc 
\ar[rrr]^{\bigoplus_{i} \Delta_i(k) e^{-k z \partial_\lambda}} 
&  & & 
H_\hT^*(X^T)_\loc 
}
\]
where $e^{-k z\partial_\lambda}$ acts on 
$\Frac(H_\hT^*(\pt))= \C(\lambda,z)$ 
by the shift of equivariant parameters $f(\lambda,z) \mapsto 
f(\lambda - k z, z)$. 
\end{definition} 

\begin{proposition}[\cite{BMO:Springer, Maulik-Okounkov}, 
{\cite[Theorem 3.14, Corollaries 3.15, 3.16]{Iritani:shift}}]  
\label{prop:shift} 
Let $k,l$ be semi-negative cocharacters of $T$. 
We have the following properties. 
\begin{enumerate}
\item $M(\tau,z) \circ \bS_k(\tau) = \cS_k \circ M(\tau,z)$ 
\item The shift operators commute with the quantum connection, 
i.e.~$[\nabla_\alpha, \bS_k(\tau)]=0$ for any $\alpha \in H^*_T(X)$.
\item We have $\bS_k(\tau) \circ \bS_l(\tau) = 
Q^{d(k,l)} \bS_{k+l}(\tau)$, 
$\cS_k \circ \cS_l = Q^{d(k,l)} \cS_{k+l}$ 
for some $d(k,l)\in H_2(X,\Z)$ which is symmetric in $k$ and $l$; 
in particular 
the shift operators commute each other: $[\bS_k(\tau), \bS_l(\tau)] 
= [\cS_k,\cS_l]=0$. 
\end{enumerate} 
\end{proposition} 

We give an explicit description for $d(k,l)$ in the above proposition. 
Let $ET \to BT$ denote a universal $T$-bundle with 
$ET \cong (\C^\infty\setminus \{0\})^{\dim T}$ and 
$BT \cong (\PP^\infty)^{\dim T}$. 
Consider the Borel construction $X_T = X \times_T ET$ of $X$. 
This is an $X$-bundle over $BT$.  
A cocharacter $k$ of $T$ induces a map 
$\varphi_k \colon \PP^1 \subset \PP^\infty \cong B\C^\times \to BT$ and 
the $X$-bundle $E_k$ can be naturally identified 
with the pull-back $\varphi_k^* X_T$ of the Borel construction. 
Therefore we have a natural map $E_k \to X_T$. 
Using this, we can compare section classes in $H_2^\sec(E_k)$ for various $k$ 
in the single space $H_2^T(X) = H_2(X_T)$. Note that   
$H_2^\sec(E_k)$ becomes a subgroup of $H_2^T(X)$ 
by the equivariant formality \cite[Proposition 2.1]{Iritani:shift} 
of $X$. 
We have the following lemma: 
\begin{lemma} 
\label{lem:d(k,l)} 
We have $d(k,l) = \sigma_{\min}(k+l) - \sigma_{\min}(k) - 
\sigma_{\min}(l)$ in $H_2^T(X)$. 
\end{lemma} 
\begin{proof} 
A straightforward calculation shows that 
$\cS_k \circ \cS_l = Q^{d(k,l)} \cS_{k+l}$ for some $d(k,l) 
\in H_2(X)$ satisfying 
\[
(\sigma_i(k) - \sigma_{\min}(k)) + (\sigma_i(l) - \sigma_{\min}(l)) 
= (\sigma_i(k+l) - \sigma_{\min}(k+l)) + d(k,l)
\] 
for each fixed component $F_i\subset X^T$, where $\sigma_i(k) 
\in H_2^\sec(E_k)$ denotes the section class associated to 
$F_i$ as in Definition \ref{def:shift_Givental}. 
We may regard this as a relation in $H_2(X_T)$ by pushing it 
forward along the inclusion $X \hookrightarrow X_T$. 
Since a fixed point in $F_i$ defines a section of the Borel construction $X_T$, 
it follows that $\sigma_i(k) + \sigma_i(l) = \sigma_i(k+l)$ in 
$H_2(X_T)$. The conclusion follows immediately. 
\end{proof} 

\begin{definition}[Seidel elements]  
Let $k\colon \C^\times \to T$ be a semi-negative cocharacter. 
The Seidel elements are defined as 
$S_k(\tau) := \lim_{z\to 0} \bS_k(\tau) 1$. 
\end{definition} 
By part (2) of Proposition \ref{prop:shift},  
$\lim_{z\to 0} \bS_k(\tau)$ commutes with 
the quantum multiplication and therefore we have 
\[
\lim_{z\to 0} \bS_k(\tau) =  (S_k(\tau) \star_\tau).  
\] 
By part (3) of Proposition \ref{prop:shift}, we have 
$S_k(\tau) \star_\tau S_l(\tau) = Q^{d(k,l)} S_{k+l}(\tau)$. 
This gives the \emph{Seidel representation} 
\cite{Seidel:pi1,McDuff-Tolman,Iritani:shift} 
of the monoid of semi-negative cocharacters on 
equivariant quantum cohomology. 

\begin{definition}[commuting vector fields {\cite[\S 4.3]{Iritani:shift}}] 
\label{def:commuting} 
For a semi-negative cocharacter $k$ of $T$, we define 
a vector field $\bV_k$ on 
$H^*_T(X) \times H^*_\hT(X) = H^*_T(X) \times 
H^*_T(X)[z]$ by 
\[
(\bV_k)_{\tau,\Upsilon}  = (S_k(\tau), [z^{-1} \bS_k(\tau)]_+ \Upsilon) 
\]
where $(\tau,\Upsilon) \in H_T^*(X) \times H^*_\hT(X)$ 
and $[z^{-1} \bS_k(\tau)]_+\Upsilon := z^{-1} \bS_k(\tau) \Upsilon- 
z^{-1} S_k(\tau) \star_\tau \Upsilon$. The vector fields $\bV_k$ 
commute each other: $[\bV_k,\bV_l]= 0$. 
\end{definition} 
\begin{remark} 
\label{rem:Givental_cone} 
The fundamental solution $M(\tau,z)$ in \eqref{eq:fundsol} defines a map 
\[
H_T^*(X) \times H^*_\hT(X) \longrightarrow H^*_\hT(X)_\loc, 
\quad (\tau,\Upsilon) 
\longmapsto z M(\tau,z) \Upsilon. 
\]
The image of this map is known as the \emph{Givental cone}  
\cite{Givental:symplectic}. Under this map, 
the vector field $\bV_k$ corresponds to the linear vector field 
on $H_\hT(X)_\loc$ given by $\bbf \mapsto z^{-1}\cS_k \bbf$ 
(see \cite[\S 4.3]{Iritani:shift}). 
The commutativity of the vector fields $\bV_k$ follows 
by Proposition \ref{prop:shift}, (3). 
\end{remark} 

\section{Equivariant mirrors of toric manifolds} 
\label{sec:mirrors} 
\subsection{Toric manifolds} 
We collect basic definitions and facts about 
toric manifolds, for which we refer 
the reader to \cite{Oda:toric,CLS}. 
Let $\bN$ be a free abelian group of finite rank. 
Consider a rational simplicial fan $\Sigma$ 
in $\bN_\R = \bN\otimes\R$ such that 
\begin{itemize} 
\item each cone of $\Sigma$ is generated by part of 
a $\Z$-basis of $\bN$; 
\item the support $|\Sigma|=\bigcup_{\sigma \in \Sigma} \sigma$ 
of the fan $\Sigma$ is full-dimensional and convex; 
\item there exists a strictly convex piecewise linear function 
$f \colon |\Sigma| \to \R$ which is linear on each cone 
of $\Sigma$. 
\end{itemize} 
These conditions ensure that the corresponding toric 
variety $X_\Sigma$ is smooth and satisfies the conditions 
in \S \ref{sec:generalities}. 
Let $b_1,\dots,b_m\in \bN$ denote primitive generators of 
the one-dimensional cones of $\Sigma$. 
These define the fan sequence 
\begin{equation} 
\label{eq:fanseq}
\begin{CD} 
0@>>> \LL @>>> \Z^m @>{(b_1,\dots,b_m)}>> \bN @>>> 0
\end{CD} 
\end{equation} 
where the third arrow sends the standard basis $e_i\in \Z^m$ 
to $b_i \in \bN$ and $\LL$ is the kernel of $\Z^m \to \bN$. 
For a subset $I\subset \{1,\dots,m\}$, we write $\sigma_I$ 
for the cone generated by $\{b_i : i \in I\}$. 
Define $K:= \LL \otimes \C^\times$. Since $\LL$ is a subgroup 
of $\Z^m$, $K$ is a subgroup of $(\C^\times)^m$ and thus 
$K$ acts on $\C^m$. 
The toric variety $X_\Sigma$ 
is defined as the quotient 
\[
X_\Sigma = (\C^m \setminus Z)/K 
\]
where $Z\subset \C^m$ is defined as the zero set 
of the ideal 
% whose defining ideal is: 
% \[
% I(Z) = \left\langle z_{i_1} z_{i_2} \cdots z_{i_k} : 
% \begin{array}{l} 
% \text{$i_1<i_2 < \cdots < i_k$, the cone spanned by} \\ 
% \text{$\{b_j : j\notin \{i_1,\dots,i_k\}\}$ 
% belongs to the fan $\Sigma$}
% \end{array} 
% \right\rangle.  
% \]
generated by monomials $\prod_{1\le i\le m, i\notin I} z_i$ 
with $I\subset \{1,\dots,m\}$ such that the cone 
$\sigma_{I}$ belongs to $\Sigma$. 
Here $z_1,\dots,z_m$ are the standard co-ordinates on $\C^m$. 
The torus $T := (\C^\times)^m/K$ naturally acts on $X_\Sigma$.  
By tensoring the exact sequence \eqref{eq:fanseq} with $\C^\times$,  
we find $T \cong \bN\otimes \C^\times$. 
In particular the lattice $\Hom(\C^\times, T)$ 
of cocharacters is identified with $\bN$. 
The toric variety $X_\Sigma$ contains the torus 
$T =(\C^\times)^m/K$, and a character $\chi \in 
\Hom(T,\C^\times) = \bN^*$ of $T$ extends to 
a regular function on $X_\Sigma$ if and only if 
$\chi \cdot v \ge 0$ for all $v\in |\Sigma| \subset \bN\otimes \R$. 
The space $H^0(X_\Sigma, \cO)$ of regular functions 
is generated by such characters, and therefore we find 
that a cocharacter $k\in \bN$ of $T$ 
is semi-negative\footnote
{Recall the convention on the $T$-action on $H^0(X_\Sigma,\cO)$ at 
the beginning of \S \ref{subsec:shift}.} if and only if $k\in |\Sigma|$. 

\begin{notation} 
\label{nota:Psi}
For $k\in \bN \cap |\Sigma|$, take a cone $\sigma_I \in \Sigma$ 
containing $k$ and write $k = \sum_{i\in I} n_i b_i$. Define 
a vector $\Psi(k) = (\Psi_1(k),\dots,\Psi_m(k)) \in (\Z_{\ge 0})^m$ as 
\[
\Psi_i(k) = \begin{cases} n_i & \text{if $i\in I$;} \\ 
0 & \text{otherwise.} 
\end{cases} 
\]
and set $|k| := \sum_{i\in I} n_i$. 
\end{notation} 

For $1\le i\le m$, let $u_i\in H^2_T(X_\Sigma)$ denote the 
Poincar\'{e} dual of the $T$-invariant divisor $\{z_i = 0\}/K \subset X_\Sigma$. 
The $T$-equivariant cohomology ring of $X_\Sigma$ is generated by 
$u_1,\dots,u_m$ over $\C$ and has the following presentation: 
\[
H^*_T(X_\Sigma)= \C[u_1,\dots,u_m]/\frI_{\rm SR}
\]
where $\frI_{\rm SR}$ is the ideal generated by $\prod_{i\in I} u_i$ such that 
the cone $\sigma_I$ does not belong to $\Sigma$. 
An element $\chi \in H^2_T(\pt,\Z) \cong \bN^*$ can be expressed 
as a linear combination of $u_i$'s: 
\begin{equation}
\label{eq:equiv_parameter}
\chi = \sum_{i=1}^m (\chi \cdot b_i) u_i.
\end{equation}
For each $k\in \bN \cap |\Sigma|$, 
define $\phi_k:= \prod_{i=1}^m u_i^{\Psi_i(k)} \in H^*_T(X_\Sigma)$. 
The following lemma is obvious from the above
presentation of $H^*_T(X_\Sigma)$. 
\begin{lemma}
\label{lem:basis} 
The set $\{\phi_k : k \in \bN \cap |\Sigma|\}$ is a basis of 
$H_T^*(X_\Sigma)$ over $\C$.  
\end{lemma} 

We have that 
$H_T^2(X_\Sigma,\Z)$ is a free $\Z$-module with 
basis $u_1,\dots,u_m$. In particular the equivariant homology 
$H^T_2(X_\Sigma,\Z)$ is identified with $\Z^m$ via the dual
basis of $u_1,\dots,u_m$. Moreover the fan sequence 
\eqref{eq:fanseq} is identified with: 
\[
\begin{CD} 
0@>>> H_2(X_\Sigma,\Z) \cong\LL @>>> H_2^T(X_\Sigma,\Z) 
\cong\Z^m @>>> H_2^T(\pt,\Z) @>>> 0.  
\end{CD} 
\]
For a cone $\sigma_I \in \Sigma$, we set 
$C_I = \{d \in H_2(X_\Sigma,\R): d \cdot u_i \ge 0 \text{ for}\ i\notin I\}$.
Then the cone of effective curves is generated by $C_I$ with $\sigma_I \in \Sigma$ 
and we have: 
\begin{equation} 
\label{eq:eff_cone} 
\Eff(X_\Sigma) = \LL \cap \sum_{\sigma_I \in \Sigma} C_I. 
\end{equation} 
Let $k\in \bN \cap |\Sigma|$ be a semi-negative cocharacter 
of $T$ and let $E_k$ be the associated $X_\Sigma$-bundle 
as in \eqref{eq:associated_bundle}.

\begin{lemma} 
\label{lem:minimal_section} 
The minimal section class $\sigma_{\min}(k)$ of 
$E_k$ is identified with the element $-\Psi(k)\in 
\Z^m \cong H_2^T(X_\Sigma,\Z)$ under the natural inclusion 
$H_2^\sec(E_k) \hookrightarrow H_2^T(X_\Sigma)$ described in the 
paragraph after Proposition \ref{prop:shift}. 
\end{lemma} 
\begin{proof} 
Let $\sigma_I\in \Sigma$ be a cone containing $k$ 
and write $k = \sum_{i\in I} n_i b_i$. 
The minimal section $\sigma_{\min}(k)$ of $E_k$ 
is associated to a point in the toric subvariety 
$\bigcap_{i\in I}\{z_i = 0\}$ whose normal bundle 
has $\C^\times$ weights $\{n_i\}_{i\in I}$. 
The class $u_i\in H^2_T(X_\Sigma)$ 
corresponds, via the natural map 
$E_k \to (X_\Sigma)_T$, to the toric divisor $\cD_i = 
\{z_i = 0\} \times_{\C^\times} 
(\C^2 \setminus \{0\})$ in $E_k$. 
It suffices to compute the intersection number of $\sigma_{\min}(k)$ 
and $\cD_i$. It is easy to see that $\cD_i \cdot \sigma_{\min}(k)$ 
equals $-n_i$ if $i\in I$ and zero otherwise. The conclusion follows. 
\end{proof} 

The above lemma and Lemma \ref{lem:d(k,l)} imply:  

\begin{corollary} 
\label{cor:d(k,l)_toric}
The class $d(k,l)\in H_2(X_\Sigma,\Z)$ in Proposition \ref{prop:shift} 
is given by $d(k,l) = \Psi(k) + \Psi(l) - \Psi(k+l)$ under 
the inclusion $H_2(X_\Sigma,\Z) \hookrightarrow H_2^T(X_\Sigma,\Z) 
\cong \Z^m$. 
In particular $d(k,l) \in \Eff(X_\Sigma)$ by \eqref{eq:eff_cone}. 
\end{corollary} 

\begin{corollary} 
\label{cor:Sk} 
We have $\bS_k(\tau) = \prod_{i=1}^m \bS_{b_i}(\tau)^{\Psi_i(k)}$. 
\end{corollary} 
\begin{proof} 
By Corollary \ref{cor:d(k,l)_toric}, $d(k,l) =0$ whenever $k$ and $l$ 
belong to the same cone. 
The conclusion follows by 
the property $\bS_k(\tau) \circ \bS_l(\tau) = Q^{d(k,l)}\bS_{k+l}(\tau)$. 
\end{proof}

\subsection{Mirror map} 
\label{subsec:mirrormap}
We introduce an infinite set 
$\by = \{y_k : k \in \bN\cap |\Sigma|\}$ of variables 
which forms a natural co-ordinate system of the B-model. 
We set $y_i := y_{b_i}$ for $1\le i\le m$ 
and $G : = (\bN \cap |\Sigma|) \setminus \{b_1,\dots,b_m\}$. 
By abuse of notation, 
we write $M[\![\by]\!]$ for the space of formal power series in the variables 
$\{ \log y_{1},\dots, \log y_{m}\} \cup \{ y_k : k\in G\}$ with coefficients 
in $M$ (see \S\ref{subsec:power_series}). 
We consider the lattice of infinite rank:
\[
\hLL = \left\{ 
\ell  = (\ell_k)_{k \in \bN\cap |\Sigma|} 
\in \Z^{\oplus \bN\cap|\Sigma|} : \sum_{k \in \bN \cap |\Sigma|} \ell_k k =0
\right\}.  
\]
By the inclusion $\{b_1,\dots,b_m\} \subset \bN\cap |\Sigma|$, 
we can regard $\LL$ as a sublattice of $\hLL$ and
we have $\hLL/\LL \cong \Z^{\oplus G}$. 
We define a splitting $\hLL \to \LL\cong H_2(X_\Sigma,\Z)$, 
$\ell \mapsto d(\ell)$ by 
\[
u_i \cdot d(\ell) = \sum_{k \in \bN\cap |\Sigma|} \ell_k \Psi_i(k).
\]
Note that we have the linear relation
$\sum_{i=1}^m (u_i \cdot d(\ell)) b_i =0$. 
We set 
\begin{align*} 
\hLL_\eff & := \{\ell \in \hLL : d(\ell) \in \Eff(X_\Sigma), \ 
\ell_k \ge 0 \ (\forall k \in G)\} \\
& \cong \Eff(X_\Sigma) \times (\Z_{\ge 0})^{\oplus G}
\end{align*} 
where in the second line we used the splitting 
$\hLL \cong \LL \oplus \Z^{\oplus G}$. 
We set, for $\ell \in \Z^{\oplus(\bN \cap |\Sigma|)}$, 
\[
\by^\ell := \prod_{\ell\in \bN\cap|\Sigma|} y_k^{\ell_k} = 
\exp\left(\sum_{i=1}^m \ell_{b_i} \log y_i\right) \prod_{k\in G} y_k^{\ell_k}.  
\] 
Note that every neighbourhood of $0$ in $\C[\![Q]\!][\![\by]\!]$ contains 
all but finite elements of $\{\by^{\ell} Q^{d(\ell)} : \ell \in \hLL_\eff\}$, 
thus a power series of the form $\sum_{\ell \in \hLL_\eff} a_\ell \by^{\ell} 
Q^{d(\ell)}$ is well-defined. 
We introduce mirror maps as an integral submanifold of 
the commuting vector fields $\{\bV_k: k\in \bN\cap |\Sigma|\}$ 
from Definition \ref{def:commuting}. 
\begin{proposition} 
\label{prop:mirror_map}
There exist unique functions 
\[
\tau(\by) \in H^*_T(X_\Sigma)[\![Q]\!] [\![\by]\!], 
\quad 
\Upsilon(\by,z) \in H^*_\hT(X_\Sigma)[\![Q]\!][\![\by]\!] 
\]
of the form 
\begin{align}
\label{eq:tau_Upsilon_expansion}  
\begin{split} 
\tau(\by) & = \sum_{i=1}^m u_i \log y_i+ 
\sum_{\ell\in \hLL_\eff \setminus \{0\}} 
\tau_\ell \by^{\ell} Q^{d(\ell)}  \\
\Upsilon(\by,z) & = 1 + 
\sum_{\ell \in \hLL_\eff \setminus\{0\}} \Upsilon_\ell(z) \by^{\ell} Q^{d(\ell)} 
\end{split} 
\end{align} 
with $\tau_\ell\in H^*_T(X_\Sigma)$, $\Upsilon_\ell(z) 
\in H^*_\hT(X_\Sigma)$ such that we have 
\begin{equation} 
\label{eq:flow} 
\parfrac{\tau(\by)}{y_k} = S_k(\tau(\by)), \qquad
\parfrac{\Upsilon(\by,z)}{y_k}= [z^{-1}\bS_k(\tau(\by))]_+ \Upsilon(\by,z)
\end{equation} 
for all $k\in \bN \cap |\Sigma|$. 
We call the function $\by \mapsto \tau(\by)$ the \emph{mirror map}. 
\end{proposition} 
\begin{proof} 
The existence and uniqueness of $\tau(\by)$, $\Upsilon(\by,z)$ along 
the locus $\{y_k = 0, \ \forall k\in G\}$ is proved in 
\cite[Proposition 4.7]{Iritani:shift}. 
Since the vector fields $\bV_k$ commute each other, we have a 
unique solution $(\tau(\by),\Upsilon(\by,z))$ to \eqref{eq:flow} which takes the form 
\eqref{eq:tau_Upsilon_expansion} along $\{y_k=0, \ \forall k\in G\}$ 
(see Theorem \ref{thm:formal_flow}). 
It suffices to show that $\tau(\by)$, $\Upsilon(\by,z)$ are expanded  
as in \eqref{eq:tau_Upsilon_expansion}. 
Write $\tau(\by) = \sum_{i=1}^m u_i \log y_i + \tau'(\by)$. 
By using the divisor equation \cite[Remark 3.12]{Iritani:shift}, we find that 
\[
\bS_k(\tau(\by);Q) = y^{-\Psi(k)} \bS_k(\tau'(\by); Qy) 
\]
where $y^{-\Psi(k)} = \prod_{i=1}^m y_i^{-\Psi_i(k)}$ 
and $\bS_k(\sigma;Qy)$ is obtained from $\bS_k(\sigma;Q)$ 
by replacing $Q^d$ with $Q^d y_1^{u_1\cdot d} \cdots y_m^{u_m\cdot d}$. 
The differential equation for $\tau$ reads: 
\[
y^{\Psi(k)}\parfrac{\tau'(\by)}{y_k} = S_k(\tau'(\by); Q y) 
\]
Notice that $e_k - \Psi(k)$ belongs to $\hLL_\eff$ and $d(e_k-\Psi(k))=0$, 
where $e_k \in \Z^{\oplus(\bN \cap |\Sigma|)}$ denotes the standard basis vector 
whose $l$th component is $\delta_{k,l}$. This shows, 
by induction on powers of the variables $\{y_k : k\in G\}$, 
that $\tau(\by)$ has an expansion of the form \eqref{eq:tau_Upsilon_expansion}. 
A similar argument shows that $\Upsilon(\by,z)$ also has 
an expansion of the form \eqref{eq:tau_Upsilon_expansion}. 
\end{proof} 

\begin{remark} 
\label{rem:linear_relation} 
By the definition of $\hLL$, 
$\tau(\by)$ and $\Upsilon(\by,z)$ satisfy the following 
equations:
\begin{align*} 
\sum_{k \in \bN\cap |\Sigma|} k \otimes y_k \parfrac{\tau(\by)}{y_k} & = 
\sum_{i=1}^m b_i \otimes u_i 
&& \text{in $\bN \otimes_\Z H^*_T(X_\Sigma)$,} 
\\
\sum_{k\in\bN \cap |\Sigma|} k \otimes y_k \parfrac{\Upsilon(\by,z)}{y_k} 
& = 0 
&& \text{in $\bN \otimes_\Z H^*_\hT(X_\Sigma)$.} 
\end{align*} 
Contracting the first equation 
with $\chi \in \bN^* \otimes \C \cong H^2_T(\pt)$ 
and using \eqref{eq:equiv_parameter}, we obtain 
\[
\sum_{k\in \bN \cap |\Sigma|} (\chi \cdot k) y_k 
\parfrac{\tau(\by)}{y_k} = \sum_{k\in \bN\cap |\Sigma|} 
(\chi \cdot k) y_k S_k(\tau(\by)) = \chi. 
\]
This is a generalization of the linear relation for Batyrev elements 
\cite{Gonzalez-Iritani:Selecta}. 
\end{remark}

\begin{lemma} 
\label{lem:classical_shift} 
The classical limit $Q\to 0$ of the shift operator $\bS_k(\tau)$ is
given by 
\[
\lim_{Q\to 0} \bS_k(\tau) f(u) = 
e^{(\tau(u -\Psi(k)z) -\tau(u))/z} 
f(u - \Psi(k)z)
\prod_{i=1}^m \prod_{c=0}^{\Psi_i(k)-1} (u_i - cz)
\]
where $f(u), \tau = \tau(u)
\in H^*_T(X_\Sigma)$ are equivariant cohomology classes  
expressed as polynomials in $u_1,\dots,u_m$ 
and $u- \Psi(k) z = (u_1 -\Psi_1(k)z, \dots, u_m - \Psi_m(k)z)$. 
In particular we have $
\lim_{Q\to 0} S_k(\tau) = 
\phi_k \exp(-\sum_{i=1}^m\Psi_i(k) \parfrac{\tau(u)}{u_i})$. 
\end{lemma} 
\begin{proof} 
This is proved when $k= b_i$ in \cite[Lemma 4.5]{Iritani:shift}. 
Note that we considered a redundant $(\C^\times)^m$-action on $X_\Sigma$ 
in \cite{Iritani:shift} and there is some difference in notation. 
The conclusion follows from Corollary \ref{cor:Sk} 
and the case where $k=b_i$. 
\end{proof}

We consider the co-ordinates 
$\bt=\{t_k:k\in \bN\cap |\Sigma|\}$ on the equivariant 
cohomology $H_T^*(X_\Sigma)$ given by 
\begin{equation} 
\label{eq:bt} 
\bt \mapsto \tau = \sum_{k\in \bN \cap |\Sigma|} t_k \phi_k 
\in H^*_T(X_\Sigma) 
\end{equation} 
where $\{\phi_k\}$ is the basis in Lemma \ref{lem:basis}. 
We define the formal neighbourhood of the origin of $H_T^*(X_\Sigma)[\![Q]\!]$ 
to be $\Spf(\C[\![Q]\!][\![\bt]\!])$. 
The mirror map identifies $\Spf(\C[\![Q]\!][\![\bt]\!])$ with 
the formal neighbourhood $\Spf (\C[\![Q]\!][\![\by]\!])$ 
of the base point $\ybase$ in the $\by$-parameter space. 
\begin{equation}
\label{eq:y_origin}
\ybase := \{y_1  = \cdots  = y_m =1, \ y_k = 0 \ (\forall k\in G)\}.  
\end{equation}

\begin{lemma} 
\label{lem:mirror_map} 
The mirror map $\by \mapsto \tau(\by)$ in Proposition \ref{prop:mirror_map} 
defines an isomorphism 
between $\Spf( \C[\![Q]\!][\![\by]\!])$ and $\Spf(\C[\![Q]\!][\![\bt]\!])$. 
\end{lemma} 
\begin{proof} 
Expand the mirror map as 
$\tau(\by) = \sum_{k\in \bN\cap |\Sigma|} t_k(\by) \phi_k$. 
One can check using the expansion \eqref{eq:tau_Upsilon_expansion} 
that $\lim_{|k|\to \infty} t_k(\by) = 0$ 
in the topology of $\C[\![Q]\!][\![\by]\!]$ (see \S\ref{subsec:power_series}). 
We also have $t_k(\by) |_{\by= \ybase, Q=0} = 0$. 
Therefore the map $\by \mapsto \tau(\by)$ defines a well-defined 
morphism $\Spf \C[\![Q]\!][\![\by]\!]\to \Spf \C[\![Q]\!][\![\bt]\!]$ 
of formal schemes. 
By Lemma \ref{lem:classical_shift}, we have  
\begin{equation} 
\label{eq:tau_Jacobi}
\left. \parfrac{\tau(\by)}{y_k} \right|_{\by = \ybase, Q=0} 
= \lim_{Q\to 0} S_k(0) = \phi_k \qquad 
\text{for } k \in \bN \cap |\Sigma|. 
\end{equation} 
The conclusion follows by the formal inverse function theorem 
(Theorem \ref{thm:formal_inverse_function}). 
\end{proof} 

\begin{remark} 
The formal neighbourhood $\Spf \C[\![Q]\!][\![\bt]\!]$ 
of the origin in $H^*_T(X_\Sigma)[\![Q]\!]$ is an infinite dimensional 
non-Noetherian and non-adic formal scheme. 
All formal schemes in this paper are the formal spectra of 
admissible rings (in the sense of \cite[Chapter 1, \S 10]{EGA:1960}), 
and thus these spaces make sense in Grothendieck's theory. 
See \cite{McQuillan:formalformal, Yasuda:non-adic} 
for more recent approaches to non-Noetherian 
formal schemes. 
\end{remark} 

\begin{lemma} 
\label{lem:homogeneity}
The cohomology classes $\tau_\ell$, $\Upsilon_\ell(z)$ 
appearing in equation \eqref{eq:tau_Upsilon_expansion} are homogeneous. We have 
$\deg \tau_\ell = 2 ( 1 - \sum_{k\in \bN\cap|\Sigma|} \ell_k)$ 
and $\deg \Upsilon_\ell(z) = - 2 \sum_{k\in \bN \cap |\Sigma|} \ell_k$.  
\end{lemma} 
\begin{proof} 
Note that the lemma implies the homogeneity of $\tau(\by)$ and 
$\Upsilon(\by,z)$ with respect to the degree $\deg y_k = 2$ of 
variables 
(except for the leading term $\sum_{i=1}^m u_i \log y_i$ of $\tau(\by)$). 
We start with the homogeneity of $\bS_k(\tau)$. 
Let $\gamma_1,\dots,\gamma_s$ be classes in $H_T^*(X_\Sigma)$. 
We claim that 
for $\tau = \sum_{i=1}^m u_i \log y_i  + \sum_{j=1}^s x_j \gamma_j$, 
$\bS_k(\tau)$ is a homogeneous endomorphism of degree $0$ if we define  
the degree of $x_j$ to be $2 - \deg \gamma_j$ and the degree of 
$y_i$ to be $2$. 
Using the divisor equation \cite[Remark 3.12]{Iritani:shift}, 
we can write $(\tbS_k(\tau)\alpha,\beta)$ 
as the sum of the following terms: 
\[
\corr{\iota_{0*}\alpha,\hat{\gamma}_{j_1},\dots, \hat{\gamma}_{j_n}, 
\iota_{\infty*}\beta}_{0,n+2,d+\sigma_{\min}}^{E_k,\hT} 
\frac{Q^d}{n!} x_{j_1} \cdots x_{j_n} 
\prod_{i=1}^m y_i^{u_i\cdot d - \Psi_i(k)} 
\]
with $d\in \Eff(X_\Sigma)$, $1\le j_1,\dots,j_n \le s$ and $n\ge 0$. 
By the virtual dimension formula of the moduli space of stable maps, 
the degree of this term is 
$\deg \alpha + \deg \beta - 2 \dim X_\Sigma$, 
where we used $c_1(E_k) \cdot \sigma_{\min} = 2 -|k|$. 
Therefore the claim follows. 
The lemma follows from this claim and the recursive construction 
of $\tau_\ell$, $\Upsilon_\ell(z)$ in \cite[Proposition 4.6]{Iritani:shift} 
and in Proposition \ref{prop:mirror_map}. 
\end{proof} 

Define the Euler vector field on the $\by$-space and on 
the $\bt$-space (i.e.~$H^*_T(X_\Sigma)$)  
by the formula: 
\begin{equation} 
\label{eq:Euler} 
\cE_\by = \sum_{k\in \bN\cap |\Sigma|} y_k \parfrac{}{y_k},
\qquad 
\cE_\bt = \sum_{i=1}^m \parfrac{}{t_{b_i}} + 
\sum_{k\in \bN \cap |\Sigma|}
\left(1 - \frac{1}{2} \deg \phi_k\right) t_k \parfrac{}{t_k}. 
\end{equation} 
Note that $\frac{1}{2} \deg \phi_k = |k| := \sum_{i=1}^m \Psi_i(k)$. 
Let $\Gr_0 \colon H_\hT^*(X_\Sigma) \to 
H_\hT^*(X_\Sigma)$ be the grading operator defined 
by $\Gr_0(\alpha) = p \alpha$ for $\alpha 
\in H_\hT^{2p}(X_\Sigma)$. 
Note that $\Gr_0(z\alpha) = z \alpha+ z \Gr_0(\alpha)$. 
Lemma \ref{lem:homogeneity} together with its proof implies 
the following: 
\begin{lemma} 
\label{lem:grading}
The mirror map $\tau(\by)$ preserves the Euler vector 
field: $\tau_* \cE_\by = \cE_\bt$. Moreover, we have 
\[  
\left[\cE_\bt + \Gr_0, \bS_k(\tau) \right] =0,  
\qquad 
(\cE_\by + \Gr_0) \Upsilon(\by,z) = 0. 
\]
In particular we have 
$[\cE_\by + \Gr_0, \bS_k(\tau(\by))]= 0$. 
\end{lemma} 
% \begin{proof} 
% The relations $\tau_* \cE_\by = \cE_\bt$ and 
% $(\cE_\by + \Gr_0) \Upsilon(\by,z)=0$ follow  
% from Lemma \ref{lem:homogeneity}. 
% The relation $[\cE_\bt + \Gr_0, \bS_k(\tau) ] =0$ is also 
% shown in the proof of Lemma \ref{lem:homogeneity}. 
% \end{proof} 

\begin{remark}
\label{rem:inverse_mirror} 
The relation $\tau_* \cE_\by = \cE_\bt$ implies the following equality: 
\begin{equation} 
\label{eq:inverse_mirror_map} 
\sum_{k\in \bN \cap |\Sigma|} y_k S_k(\tau) = 
\sum_{k\in \bN \cap |\Sigma|} y_k \parfrac{\tau(\by)}{y_k} 
= c_1^T(X_\Sigma) + 
\sum_{k\in \bN \cap |\Sigma|} (1-|k|)t_k \phi_k. 
\end{equation} 
Since $\{S_k(\tau): k\in \bN\cap |\Sigma|\}$ forms a 
$\C[\![Q]\!][\![\bt]\!]$-basis of $H_T^*(X_\Sigma)[\![Q]\!][\![\bt]\!]$, 
the inverse mirror map $y_k(\bt)$ is obtained as the 
expansion coefficients of the right-hand side in $\{S_k(\tau)\}$. 
\end{remark} 

\begin{remark}[divisor equation]
% Here we did not consider the degree of the Novikov 
% variable; it is simply set to be zero. 
% In fact, the divisor equation replaces the degree 
% of $Q^d$ with part of the Euler vector field. 
Let $Q_i\parfrac{}{Q_i}$ be the differential 
operator in the Novikov variable such that 
$(Q_i\parfrac{}{Q_i}) Q^d = (u_i \cdot d ) Q^d$. 
By \eqref{eq:tau_Upsilon_expansion}, 
$\tau(\by)$ and $\Upsilon(\by,z)$ satisfy the following 
analogue of the divisor equation: 
\begin{align*} 
Q_i\parfrac{}{Q_i} \tau(\by)
= \bD_i \tau(\by) - u_i, \qquad 
Q_i\parfrac{}{Q_i} \Upsilon(\by,z)= \bD_i \Upsilon(\by,z), 
\end{align*} 
where we set $\bD_i = \sum_{k\in \bN\cap |\Sigma|} \Psi_i(k) 
y_k \parfrac{}{y_k}$.  
The divisor equation for $\bS_k(\tau)$ can be written 
in the following form: 
\begin{align*} 
Q_i \parfrac{}{Q_i} \bS_k(\tau) & = \left(
\parfrac{}{t_{b_i}} + \Psi_i(k) \right) \bS_k(\tau),  \\ 
Q_i \parfrac{}{Q_i} \bS_k(\tau(\by)) 
& = (\bD_i + \Psi_i(k)) \bS_k(\tau(\by)).   
\end{align*} 
\end{remark} 
\begin{remark} 
In this paper, we do not consider the degree of the Novikov 
variable; it is simply set to be zero. 
Some people introduce the degree 
$\deg Q^d = c_1(X_\Sigma) \cdot d$, 
which corresponds to the part $\sum_{i=1}^m\parfrac{}{t_{b_i}}$ 
of the Euler vector field $\cE_\bt$ via the divisor equation. 
\end{remark}

\subsection{Gauss-Manin system}
\label{subsec:GM}
Regarding $\Eff(X_\Sigma) \subset H_2(X,\Z) \cong \LL$ 
as a subset of $\Z^m$, we consider the semigroup 
\[
\M = \Eff(X_\Sigma) + (\Z_{\ge 0})^m. 
\]
For a ring $R$, we introduce a certain completion 
$R\{\M\}$ of the semigroup ring $R[\M]$. 
We write $w_i$ for the element of $R[\M]$ corresponding 
to the $i$th basis vector $e_i\in (\Z_{\ge 0})^m$ 
and write $Q^d$ for the element of $R[\M]$ corresponding 
to $d\in \Eff(X_\Sigma)$. The uncompleted Novikov ring 
$R[Q] := R[\Eff(X_\Sigma)]$ is naturally contained in $R[\M]$. 
Consider the ideal of $R[\M]$ generated by $Q^d$ 
with $d\in \Eff(X_\Sigma)\setminus \{0\}$ and 
write $R\{\M\}$ for the completion with respect 
to this ideal. Then $R\{\M\}$ is an $R[\![Q]\!]$-algebra. 

The mirror Landau-Ginzburg model is defined on the 
space $\Spf(\C\{\M\})$.  
We introduce a convenient co-ordinate system $(x,Q)$ on it. 
We consider the semigroup ring $\C[\bN \cap |\Sigma|]$ 
of $\bN \cap |\Sigma|$ and denote by $x^k \in \C[\bN\cap |\Sigma|]$ 
for the element corresponding to $k\in \bN \cap |\Sigma|$. 
Choose a maximal cone $\sigma_{I_0}$ of $\Sigma$. 
Since $\{b_i :i\in I_0\}$ is a $\Z$-basis of $\bN$, we can 
define a splitting $\varsigma \colon \bN \to \Z^m$ of the 
fan sequence \eqref{eq:fanseq} by sending 
$b_i\in \bN$ with $i\in I_0$ to $e_i\in \Z^m$. 
This splitting defines an embedding 
$\C[\M] \hookrightarrow \C[Q]\otimes \C[\bN \cap |\Sigma|]$ 
of $\C[Q]$-algebras by the assignment: 
\[
w_j = Q^{e_j - \varsigma(b_j)} x^{b_j}.  
\]
Note that $e_j - \varsigma(b_j)\in \Z^m$ lies in $\Eff(X_\Sigma)$ 
(see \eqref{eq:eff_cone}). 
This exhibits $\C\{\M\}$ as a subalgebra 
of $\C[\bN\cap |\Sigma|][\![Q]\!]$. 
For $k\in \bN \cap |\Sigma|$, we write 
\[
w^{\Psi(k)} := \prod_{i=1}^m w_i^{\Psi_i(k)} = Q^{\beta(k)} x^k  
\]
where $\beta(k) := \Psi(k) - \varsigma(k) \in \Eff(X_\Sigma)$.  
Then we have $\C\{\M\} = (\bigoplus_{k\in \bN\cap |\Sigma|} 
\C w^{\Psi(k)})[\![Q]\!]$, i.e.~$\{w^{\Psi(k)}:k \in \bN\cap|\Sigma|\}$ 
is a topological $\C[\![Q]\!]$-basis of $\C\{\M\}$. 
We define the \emph{universal Landau-Ginzburg potential} by 
\[
F(x;\by) := \sum_{k\in \bN\cap |\Sigma|} y_k w^{\Psi(k)} 
= \sum_{k\in \bN \cap |\Sigma|} y_k Q^{\beta(k)} x^k.  
\]
This parameterizes all elements of $\C\{\M\}$ and belongs 
to $\C\{\M\}[\![\by]\!]$. 
We also consider the equivariant version:   
\[
F_\lambda(x;\by) = F(x;\by) - \lambda \cdot \log x 
\]
where $\lambda \in \bN\otimes \C = \Lie(T)$ is an 
equivariant parameter and $\log x$ is regarded as a 
point in $\bN^* \otimes\C$. Choosing an auxiliary basis 
of $\bN$, we write $x = (x_1,\dots,x_D)$ and 
$\lambda =(\lambda_1,\dots,\lambda_D)$ 
so that $\lambda \cdot \log x = \sum_{i=1}^D \lambda_i \log x_i$ 
where $D = \rank \bN$. 

\begin{definition} 
Set $\omega = \frac{dx_1}{x_1} \cdots \frac{dx_D}{x_D}$ 
and $\omega_i = \iota(x_i\parfrac{}{x_i}) \omega$. 
The \emph{(logarithmic) Gauss-Manin system} $\GM(F_\lambda)$ of 
$F_\lambda(x;\by)$ is defined to be the cokernel of the map 
\[
z d + dF_\lambda \wedge \colon 
\bigoplus_{i=1}^D 
\C[z]\{\M\}[\![\by]\!][\lambda] \omega_i 
\to \C[z]\{\M\}[\![\by]\!][\lambda] \omega. 
\]
where $d$ is the derivation with respect to the $x$-variables 
which is linear over the ground ring $\C[z][\![Q]\!][\![\by]\!][\lambda]$ 
and defined on generators by 
$d w^{\Psi(k)} = \sum_{i=1}^D k_i Q^{\beta(k)} x^k dx_i/x_i
= \sum_{i=1}^D k_i w^{\Psi(k)} dx_i/x_i$. 
The grading operator on the Gauss-Manin system is given by: 
\[
\Gr 
= \cE_\by  + z\parfrac{}{z} + \sum_{i=1}^D \lambda_i \parfrac{}{\lambda_i}, 
\]
where $\cE_\by$ is the Euler vector field in \eqref{eq:Euler}. 
\end{definition}

\begin{remark} 
The term $\lambda \cdot \log x$ in $F_\lambda(x;\by)$ depends on 
the choice of a splitting $\varsigma \colon\bN \to \Z^m$, but 
the Gauss-Manin system $\GM(F_\lambda)$ itself does not. 
\end{remark} 

\begin{remark} 
\label{rem:oscillatory} 
Each element $f(z,x,\by)\omega\in \GM(F_\lambda)$ 
associates the following oscillatory integral: 
\[
\int e^{F_\lambda(x;\by)/z} f(z,x,\by) \omega. 
\]
The image of $z d + dF_\lambda\wedge$ corresponds to 
exact oscillatory forms. 
The cohomology with respect to the twisted de Rham differential 
$zd + dF_\lambda \wedge$ has been used in singularity theory 
\cite{SaitoK:higherresidue, Sabbah:tame} and 
also in the context of the GKZ system 
\cite{Adolphson, Borisov-Horja:bbGKZ, Reichelt-Sevenheck:logFrob, 
Reichelt-Sevenheck:nonaffine}. 
\end{remark} 

The image of $z d + dF_\lambda \wedge$ is topologically generated by 
\begin{equation} 
\label{eq:lambda_action} 
(zd + dF_\lambda \wedge) w^{\Psi(l)} \omega_i = 
\left(z l_i w^{\Psi(l)} + \sum_{k\in \bN\cap |\Sigma|} k_i y_k w^{\Psi(k)+\Psi(l)} 
- \lambda_i w^{\Psi(l)}\right) \omega  
\end{equation} 
as a $\C[z][\![Q]\!][\![\by]\!][\lambda]$-module 
where $k_i$, $l_i$ denote the $i$th components of $k,l\in \bN\cap |\Sigma|$ 
(with respect to the auxiliary basis of $\bN$). 
This gives a relation in $\GM(F_\lambda)$ and 
can be thought of as \emph{defining} the action 
of $\lambda_i$. Therefore, we have:  
\[
\GM(F_\lambda) \cong \C[z]\{\M\}[\![\by]\!] \omega.  
\] 
\begin{proposition} 
There exists a unique connection
$\nabla_{\parfrac{}{y_k}} \colon\GM(F_\lambda) 
\to z^{-1} \GM(F_\lambda)$ in the $\by$-direction  
which satisfies 
\begin{align*} 
\nabla_{\parfrac{}{y_k}} (f s) & = 
\parfrac{f}{y_k} s + f \nabla_{\parfrac{}{y_k}} s,  \\ 
\left[\nabla_{\parfrac{}{y_k}}, \nabla_{\parfrac{}{y_l}}\right] & =0, 
\qquad 
\left[\nabla_{\parfrac{}{y_k}}, \Gr\right] = \nabla_{\parfrac{}{y_k}}, 
\\
\nabla_{\parfrac{}{y_k}} w^{\Psi(l)}\omega &= z^{-1} \parfrac{F_\lambda}{y_k} 
w^{\Psi(l)} \omega
= z^{-1} w^{\Psi(k) + \Psi(l)} \omega
\end{align*} 
for $f\in \C[z][\![Q]\!][\![\by]\!][\lambda]$, $s\in \GM(F_\lambda)$, 
$k,l\in \bN\cap |\Sigma|$. We call $\nabla$ the \emph{Gauss-Manin 
connection}. 
\end{proposition} 
\begin{proof} 
Since $\{w^{\Psi(k)}: k\in \bN\cap |\Sigma|\}$ is a topological 
$\C[z][\![Q]\!][\![\by]\!][\lambda]$-basis of $\C[z]\{\M\}[\![\by]\!][\lambda]$, 
we have a unique connection $\nabla_{\parfrac{}{y_k}}$ on 
$\C[z]\{\M\}[\![\by]\!][\lambda]\omega$ 
satisfying the above properties. 
It suffices to check that $z\nabla_{\parfrac{}{y_k}}$ 
preserves the image of $(z d + dF_\lambda\wedge)$. 
Using Corollary \ref{cor:d(k,l)_toric}, we find that 
\[
z \nabla_{\parfrac{}{y_k}}\left [ (zd + dF_\lambda\wedge) 
w^{\Psi(l)}\omega_i \right] = Q^{d(k,l)} (zd + dF_\lambda \wedge) 
w^{\Psi(k+l)} \omega_i. 
\]
The conclusion follows. 
\end{proof} 

We can introduce shift operators for the Gauss-Manin system. 
\begin{proposition} 
\label{prop:shift_GM} 
Consider the operator $w^{\Psi(k)}\colon \GM(F_\lambda) 
\to\GM(F_\lambda)$ with $k\in \bN\cap |\Sigma|$ 
given by multiplication by $w^{\Psi(k)}$ 
on $\C[z]\{\M\}[\![\by]\!]\omega \cong \GM(F_\lambda)$. 
This satisfies the following properties: 
\begin{align*}
w^{\Psi(k)} \circ w^{\Psi(l)} & = w^{\Psi(l)} \circ w^{\Psi(k)} = 
Q^{d(k,l)} w^{\Psi(k+l)} \\
w^{\Psi(k)} \circ f(Q,\by,\lambda) &= f(Q,\by,\lambda-kz) \circ w^{\Psi(k)} 
\\
w^{\Psi(k)} \circ \nabla_{\parfrac{}{y_k}} 
&= \nabla_{\parfrac{}{y_k}} \circ w^{\Psi(k)}, 
\qquad 
w^{\Psi(k)} \circ \Gr = \Gr \circ w^{\Psi(k)}, 
\end{align*} 
where $k,l\in \bN\cap |\Sigma|$ and 
$f(Q,\by,\lambda) \in \C[\![Q]\!][\![\by]\!][\lambda]$. 
\end{proposition} 
\begin{proof} 
The first equation follows from Corollary \ref{cor:d(k,l)_toric}.  
The second equation follows from the relation \eqref{eq:lambda_action} 
defining the action of $\lambda_i$. The third and the fourth are obvious. 
\end{proof} 

We now state our main result. 
Recall that $\bt = \{t_k : k \in \bN \cap |\Sigma|\}$ 
is a co-ordinate system on $H_T^*(X_\Sigma)$ given by \eqref{eq:bt} 
and the mirror map gives a formal invertible change of variables 
between $\by$ and $\bt$ (Lemma \ref{lem:mirror_map}). 

\begin{theorem} 
\label{thm:mirror_isom} 
We identify the parameters $\by=\{y_k\}$ 
and $\bt = \{t_k\}$ 
via the mirror map in Proposition \ref{prop:mirror_map}. 
Then we have a unique isomorphism 
\[
\Theta \colon \GM(F_\lambda) \xrightarrow{\cong} 
H^*_T(X_\Sigma)[z][\![Q]\!][\![\bt]\!] 
\] 
of $\C[z][\![Q]\!][\![\by]\!]$-modules such that 
\begin{enumerate}
\item $\Theta(\omega) = \Upsilon(\by,z)$;  
\item $\Theta$ intertwines $w^{\Psi(k)}$ with $\bS_k(\tau(\by))$ for 
$k\in \bN\cap |\Sigma|$;   
\item $\Theta$ intertwines the Gauss-Manin connection with 
the quantum connection;
\item $\Theta$ intertwines the action of equivariant parameters 
$\lambda_i$, $i=1,\dots,D$; 
\item $\Theta$ preserves the grading, i.e.~$\Theta \circ \Gr = 
(\cE_\by + \Gr_0) \circ \Theta$, 
\end{enumerate} 
where $\Upsilon(\by,z)$ is as in Proposition \ref{prop:mirror_map}. 
\end{theorem} 
\begin{proof} 
Since $\{w^{\Psi(k)} \omega: k \in \bN\cap |\Sigma|\}$ gives 
a topological $\C[z][\![Q]\!][\![\by]\!]$-basis of $\GM(F_\lambda)$, 
we can define a $\C[z][\![Q]\!][\![\by]\!]$-module 
homomorphism $\Theta$ by setting 
$\Theta(w^{\Psi(k)}\omega) = \bS_k(\tau(\by)) \Upsilon(\by,z)$. 
By Lemma \ref{lem:classical_shift} and the form \eqref{eq:tau_Upsilon_expansion} 
of $\tau(\by)$, we find that 
\[
\bS_k(\tau(\by)) \Upsilon(\by,z)\big|_{\by = \ybase, Q=0} = 
\prod_{i=1}^m \prod_{c=0}^{\Psi_i(k)-1} (u_i -cz) = \phi_k + O(z) 
\] 
where $\ybase$ is given in \eqref{eq:y_origin}. Thus $\Theta$ is an isomorphism. 
This preserves the grading since 
we have $\Gr (w^{\Psi(k)}\omega) = 0$ and  
$(\cE_\by + \Gr_0) \bS_k(\tau(\by)) \Upsilon(\by,z) = 0$ 
by Lemma \ref{lem:grading}. 
Proposition \ref{prop:shift} (3) and Proposition \ref{prop:shift_GM} 
show that $\Theta$ intertwines $w^{\Psi(k)}$ with $\bS_k(\tau(\by))$. 
The differential equation \eqref{eq:flow} for $\Upsilon$ gives: 
\[
\parfrac{}{y_k} \Upsilon(\by,z) + z^{-1}
\parfrac{\tau(\by)}{y_k} \star_{\tau(\by)} \Upsilon(\by,z) 
= z^{-1}\bS_k(\tau(\by)) \Upsilon(\by,z) 
\]
where we used $\parfrac{\tau(\by)}{y_k} = S_k(\tau(\by))$. 
This implies that 
\[
\nabla^{\rm QC}_{\parfrac{}{y_k}} \Theta(\omega) = 
\Theta \left(\nabla^{\rm GM}_{\parfrac{}{y_k}} \omega \right) 
\]
where QC stands for quantum connection 
and GM stands for Gauss-Manin. 
Since $\nabla^{\rm QC}$ commutes with $\bS_k(\tau(\by))$, 
$\Theta$ intertwines $\bS_k(\tau(\by))$ with $w^{\Psi(k)}$, 
and $w^{\Psi(k)}$ commutes with $\nabla^{\rm GM}$, it 
follows that we can replace $\omega$ 
with $w^{\Psi(k)} \omega$ in this formula. 
Part (3) follows. 
In view of the relation \eqref{eq:lambda_action}, part (4) 
is equivalent to the relation: 
\[
\lambda_i \bS_l(\tau(\by))\Upsilon(\by,z) 
= z l_i \bS_l(\tau(\by)) \Upsilon(\by,z) + 
\sum_{k\in \bN\cap |\Sigma|} 
k_i y_k \bS_k(\tau(\by))\bS_l(\tau(\by)) \Upsilon(\by,z). 
\]
It suffices to show the equality for $l=0$. Using the differential 
equation \eqref{eq:flow} again, we find that this is equivalent to: 
\[
\lambda_i \Upsilon(\by,z) = \sum_{k\in \bN\cap |\Sigma|} 
z k_i y_k \parfrac{\Upsilon(\by,z)}{y_k} + 
\sum_{k\in \bN \cap |\Sigma|} 
k_i y_k \parfrac{\tau(\by)}{y_k} \star \Upsilon(\by,z). 
\]
This follows from Remark \ref{rem:linear_relation}. 
\end{proof} 

We define the Jacobi ring of $F_\lambda$ to be 
\[
J(F_\lambda) := \C\{\M\}[\![\by]\!][\lambda]/
\left(x_1\partial_{x_1} F_\lambda(x;\by), \dots, 
x_D \partial_{x_D} F_\lambda(x;\by) \right) 
\cong  \C\{\M\}[\![\by]\!]. 
\]
The following corollary gives a combinatorial description of the big 
equivariant quantum cohomology as an abstract ring. 

\begin{corollary} 
\label{cor:equiv_Jacobi_ring} 
We have a $\C[\![Q]\!][\![\by]\!][\lambda]$-algebra isomorphism 
$J(F_\lambda) \cong \C\{\M\}[\![\by]\!] \xrightarrow{\cong}  (H^*_T(X_\Sigma)[\![Q]\!][\![\bt]\!],\star_\tau)$ 
that sends $w^{\Psi(k)} = \partial_{y_k} F(x;\by)$ to 
$\parfrac{\tau(\by)}{y_k}$ for $k\in \bN \cap |\Sigma|$. 
\end{corollary} 
\begin{proof} 
Note that we have $\GM(F_\lambda)/z \GM(F_\lambda) \cong 
J(F_\lambda) \cdot \omega$. Since $\Theta$ intertwines the Gauss-Manin 
connection with the quantum connection, it induces  
an isomorphism $J(F_\lambda) \cdot \omega 
\cong H^*_T(X_\Sigma)[\![Q]\!][\![\bt]\!]$ 
intertwining the action of $w^{\Psi(k)}$ with the 
quantum product $\parfrac{\tau(\by)}{y_k}\star_{\tau(\by)}$. 
\end{proof} 

\subsection{Primitive form} 
\label{subsec:primitive_form} 
We define the primitive form $\zeta\in \GM(F_\lambda)$ as 
the inverse image of the identity class $1$ 
\[
\zeta := \Theta^{-1}(1) 
= \sum_{k\in \bN\cap |\Sigma|}c_k(z,\by) w^{\Psi(k)} \omega
\]
under the isomorphism $\Theta$ in Theorem \ref{thm:mirror_isom}, 
where $c_k(z,\by) \in \C[z][\![Q]\!][\![\by]\!]$. 
Since $\Theta$ intertwines the Gauss-Manin connection 
with the quantum connection, we obtain: 

\begin{proposition}
When we identify the parameters $\by$ and $\bt$ via the mirror map, 
the primitive form satisfies the differential equation: 
\[
z \nabla_{\parfrac{}{t_k}} \nabla_{\parfrac{}{t_l}} 
\zeta = \sum_{j \in \bN\cap |\Sigma|} c_{k,l}^j(\bt) \nabla_{\parfrac{}{t_j}} \zeta 
\]
where $c_{k,l}^j(\bt)\in \C[\![Q]\!][\![\bt]\!]$ are the structure constants 
of the equivariant quantum product, i.e.~$\phi_k \star_\tau \phi_l 
= \sum_{j} c_{k,l}^j(\bt) \phi_j$. 
\end{proposition}

We give an alternative description for the mirror isomorphism $\Theta$ 
and the primitive form. Introduce an infinite set 
$\by^+ = \{ y_{k,n} : k \in \bN \cap |\Sigma| , n =1,2,3,\dots\}$ 
of parameters and consider the formal deformation 
of $\omega$: 
\begin{equation} 
\label{eq:omega_y+} 
\omega(\by^+) = \exp\left( 
\sum_{k\in \bN\cap |\Sigma|} 
\sum_{n=1}^\infty y_{k,n} z^{n-1}  w^{\Psi(k)}\right) \omega. 
\end{equation} 
In view of the oscillatory integral in Remark \ref{rem:oscillatory}, 
this formal deformation corresponds to adding to the potential $F_\lambda$ 
the $z$-dependent term $\sum_{k} \sum_{n\ge 1} y_{k,n} z^n w^{\Psi(k)}$. 
The original parameters $\by= \{y_k\}$ correspond to $\{y_{k,0}\}$ in this sense. 
Note that the primitive form $\zeta$ can be written in this form \eqref{eq:omega_y+} 
since $c_k(z,\by)|_{\by=\ybase, Q=0} = \delta_{k,0}$. 
We will work with formal power series in all these variables 
$\{\log y_1,\dots,\log y_m\} \cup \{y_k : k\in G\} \cup 
\{y_{k,n}: k\in \bN\cap |\Sigma|, n\ge 1\}$. 
The following theorem follows easily from Proposition \ref{prop:mirror_map} 
and Theorem \ref{thm:mirror_isom}. 

\begin{theorem} 
\label{thm:y+} 
The image $\Upsilon(\by,\by^+,z) = \Theta(\omega(\by^+))$ is 
characterized by the following differential equation: 
\begin{align} 
\label{eq:Upsilon_y+flow}
% \parfrac{\Upsilon(\by,\by^+,z)}{y_k} 
% & = [z^{-1} \bS_k(\tau(\by))]_+ \Upsilon(\by,\by^+,z) \\ 
\parfrac{\Upsilon(\by,\by^+,z)}{y_{k,n}} 
& = [z^{n-1}\bS_k(\tau(\by))]_+ \Upsilon(\by,\by^+,z)  \quad n=0,1,2,\dots 
\end{align} 
where we set $y_{k,0} := y_k$, $[z^{n-1} \bS_k(\tau(\by))]_+ \Upsilon 
:= z^{n-1} \bS_k(\tau(\by))\Upsilon - \delta_{n,0} z^{-1} 
S_k(\tau(\by))\star_{\tau(\by)}\Upsilon$, together 
with the expansion 
\begin{equation} 
\label{eq:expansion_Upsilon_y+}
\Upsilon(\by,\by^+,z) = 1 + \sum_{(\ell,\ell^+)} 
\Upsilon_{\ell,\ell^+}(z) 
\by^\ell  (\by^+)^{\ell^+} Q^{d(\ell,\ell^+)}
\end{equation} 
with $\Upsilon_{\ell,\ell^+}(z) \in H^*_\hT(X_\Sigma)$, 
where the sum is taken over $(\ell,\ell^+)\in 
\Z^{\oplus (\bN\cap |\Sigma|)}\times \Z^{\oplus ( (\bN\cap |\Sigma|)\times \N)}$ 
such that 
\begin{itemize} 
\item $\sum_{k\in\bN \cap |\Sigma|} (\ell_k + \sum_{n=1}^\infty 
\ell_{k,n}^+) k =0$; 
\item $\ell_k\ge 0$ for all $k\in G$, 
$\ell_{k,n}^+ \ge 0$ for all $(k,n) \in (\bN \cap |\Sigma|) \times \N$, and; 
\item  
$d(\ell,\ell^+)\in \Eff(X_\Sigma)$ 
\end{itemize} 
where $\N = \{1,2,3,\dots\}$ is the set of natural numbers and $d(\ell,\ell^+)
\in H_2(X_\Sigma,\Z)$ is 
determined by $u_i \cdot d(\ell,\ell^+) = \sum_{k\in \bN \cap |\Sigma|}
(\ell_k + \sum_{n=1}^\infty \ell^+_{k,n}) \Psi_i(k)$. 
\end{theorem} 

We consider the formal neighbourhood of $(0,1)$ in $H^*_T(X_\Sigma)[\![Q]\!] 
\times H^*_\hT(X_\Sigma)[\![Q]\!]$. This is defined similarly 
to the formal neighbourhood of $0$ in $H^*_T(X_\Sigma)[\![Q]\!]$ 
(see the discussion around \eqref{eq:bt}) 
by choosing a $\C$-linear basis of $H^*_T(X_\Sigma) \times 
H_\hT^*(X_\Sigma)$. 

\begin{lemma} 
\label{lem:very_big_mirrormap}
The map $(\by,\omega(\by^+)) \mapsto (\tau(\by), \Upsilon(\by,\by^+,z))$ 
defines an isomorphism between the formal neighbourhood 
$\Spf (\C[\![Q]\!][\![\by,\by^+]\!])$ 
of $\omega$ in the Gauss-Manin system and the formal neighbourhood of 
$(0,1)$ in $H^*_T(X_\Sigma)[\![Q]\!] \times H^*_\hT(X_\Sigma)[\![Q]\!]$. 
\end{lemma} 
\begin{proof} 
The map $(\by,\omega(\by^+)) \mapsto (\tau(\by),\Upsilon(\by,\by^+,z))$ 
defines a morphism of formal schemes for a reason similar to 
the proof of Lemma \ref{lem:mirror_map}. 
By the formal inverse function theorem (Theorem \ref{thm:formal_inverse_function}), 
it suffices to check that the differential of the map 
at $\by = \ybase$, $\by^+=0$, $Q=0$ is an isomorphism 
(see \eqref{eq:y_origin} for $\ybase$). 
We checked that $\by \mapsto \tau(\by)$ 
is an isomorphism in Lemma \ref{lem:mirror_map}. 
By Lemma \ref{lem:classical_shift}, we have for $n\ge 1$ 
\begin{equation}
\label{eq:Upsilon_Jacobi}
\left. \parfrac{\Upsilon(\by,\by^+,z)}{y_{k,n}}
\right|_{\by=\ybase,\by^+ =0, Q=0} = 
z^{n-1} \prod_{i=1}^m 
\prod_{c=0}^{\Psi_i(k)-1}(u_i- cz). 
\end{equation}
These form a $\C$-linear basis of $H^*_{\hT}(X_\Sigma)$, and 
the conclusion follows. 
\end{proof}

\begin{corollary}
\label{cor:primitive_Upsilon}
The primitive form $\zeta$ is given by $\omega(\by^+)$ 
with $\Upsilon(\by,\by^+,z) =1$. 
\end{corollary} 

\begin{remark}[cf.~Remark \ref{rem:inverse_mirror}]
Extending the argument in Lemma \ref{lem:grading}, we can 
easily show the homogeneity  
$(\cE_{\by,\by^+} + \Gr_0) \Upsilon(\by,\by^+,z) = 0$ with respect 
to the extended Euler vector field 
$\cE_{\by,\by^+} = \sum_{k\in \bN\cap |\Sigma|} 
\sum_{n=0}^\infty 
(1-n) y_{k,n} \parfrac{}{y_{k,n}}$, where we set $y_{k,0} = y_k$. 
From this we obtain 
\[
\sum_{k,n} (1-n) y_{k,n} \left[ z^{n-1} \bS_k(\tau(\by))\right]_+   
\Upsilon(\by,\by^+,z) + \Gr_0 \Upsilon(\by,\by^+,z) = 0.  
\]
Suppose that $\by^+$ is chosen so that 
$\Upsilon(\by,\by^+,z) = 1$ and $\omega(\by^+)$ is the primitive form. 
Then one obtains by \eqref{eq:inverse_mirror_map} 
that 
\[
\sum_{k\in \bN\cap |\Sigma|} \sum_{n=0}^\infty 
(1-n) y_{k,n} z^n \bS_k(\tau(\by)) 1 
= c_1^T(X_\Sigma) + \sum_{k\in \bN\cap |\Sigma|} 
(1 - |k|) t_k \phi_k. 
\]
This determines $y_{k,n}$ with $n\neq 1$ as the expansion 
coefficients of the right-hand side, as $z^n \bS_k(\tau) 1$ form 
a basis of $H_\hT^*(X_\Sigma)$. 
\end{remark} 

\begin{remark} 
In this paper, we do not study the higher residue 
pairing \cite{SaitoK:higherresidue} for the Gauss-Manin system. 
It would be interesting to define such a structure in our setting. 
Since $X_\Sigma$ is not necessarily compact, 
the higher residue pairing should take values in 
$\C(\lambda_1,\dots,\lambda_D)[z][\![Q]\!][\![\by]\!]$ 
in general and coincide with the Poincar\'{e} pairing on quantum 
cohomology. 
See e.g.~\cite[Appendix A.3]{Iritani:integral} for the comparison 
of (non-equivariant) pairings for compact weak-Fano toric orbifolds. 
\end{remark} 

% \begin{remark} 
% Since $\zeta$ is of degree zero, i.e.~$\Gr (\zeta) = 0$, we have: 
% \begin{align*} 
% z \nabla_{\cE_\by} \zeta & = 
% \sum_{k\in \bN\cap |\Sigma|} z (\cE_\by c_k(z,\by)) w^{\Psi(k)} \omega
%  + F(x,\by) \zeta \\ 
% & = - \sum_{k\in \bN\cap |\Sigma|} 
%  z^2\partial_z(c_k(z,\by)) w^{\Psi(k)} \omega + F(x,\by) \zeta.  
% \end{align*} 
% Sending this by $\Theta$ and using Lemma \ref{lem:grading}, we obtain: 
% \[
% \Theta( F(x,\by) \zeta ) \equiv c_1^T(X_\Sigma) 
% + \sum_{k\in \bN\cap |\Sigma|} \left(1-\frac{\deg\phi_k}{2}\right) t_k \phi_k 
% \mod z^2. 
% \]
% \end{remark} 

\section{Non-equivariant mirrors} 
\subsection{Non-equivariant mirror isomorphism} 
\label{subsec:noneq_mirror_isom}
We obtain a non-equivariant mirror isomorphism by taking 
the non-equivariant limit $\lambda \to 0$ of Theorem \ref{thm:mirror_isom}. 
The non-equivariant Gauss-Manin system $\GM(F)$ is defined 
to be the cokernel of the map: 
\[
zd + dF \wedge \colon 
\bigoplus_{i=1}^D \C[z]\{\M\}[\![\by]\!] \omega_i 
\to \C[z]\{\M\}[\![\by]\!] \omega 
\]
which is just $\GM(F_\lambda)/\sum_{i=1}^D \lambda_i \GM(F_\lambda)$. 
The mirror isomorphism $\Theta$ in Theorem \ref{thm:mirror_isom} induces 
an isomorphism 
\begin{equation} 
\label{eq:Theta_noneq} 
\Theta_\noneq 
\colon \GM(F) \xrightarrow{\cong} H^*(X_\Sigma)[z][\![Q]\!][\![\bt]\!]. 
\end{equation} 
In the non-equivariant limit, we can extend the flat connection 
in the $z$-direction using the homogeneity. 
For the non-equivariant Gauss-Manin system $\GM(F)$,  
we define: 
\[
\nabla_{z\parfrac{}{z}} := \Gr - \nabla_{\cE_\by}.  
\]
Explicitly this is given by 
\[
\nabla_{z\parfrac{}{z}} (f(x,z,\by)\omega) 
=  \left( z\parfrac{f(x,z,\by)}{z} -z^{-1} F(x;\by) f(x,z,\by)\right) \omega 
\]
for $f(x,z,\by) \in \C[z]\{\M\}[\![\by]\!]$. 
For the non-equivariant quantum cohomology module 
$H^*(X_\Sigma)[z][\![Q]\!][\![\bt]\!]$, 
we define: 
\[
\nabla_{z\parfrac{}{z}} := (\Gr_0 + \cE_\bt) - \nabla_{\cE_\bt}    
\]
which is given explicitly as:  
\[
\nabla_{z\parfrac{}{z}} f(z,\bt) = \Gr_0(f(z,\bt)) - \frac{1}{z} 
\left(c_1(X_\Sigma) + \sigma - \Gr_0(\sigma) \right) \star_\sigma f(z,\bt) 
\]
for $f(z,\bt) \in H^*(X_\Sigma)[z][\![Q]\!][\![\bt]\!]$,
where $\sigma$ is the non-equivariant limit of 
$\tau = \sum_{k\in \bN\cap |\Sigma|} t_k \phi_k$ 
and $\star_\sigma$ denotes the non-equivariant quantum product. 
Note that $\Gr_0$ contains the derivation $z\parfrac{}{z}$. 
These operators $\nabla_{z\parfrac{}{z}}$ have a pole of 
order one along $z=0$. By Theorem \ref{thm:mirror_isom}, 
it is clear that the isomorphism $\Theta_\noneq$ intertwines 
the quantum connection and the Gauss-Manin connection including 
in the $z$-direction.  

We further restrict the base space to the non-equivariant 
cohomology $H^*(X_\Sigma)$. 
Let $T_0,\dots,T_N$ denote a homogeneous basis of $H^*(X_\Sigma)$ 
and let $s_0,\dots,s_N$ be the co-ordinates on $H^*(X_\Sigma)$ 
dual to $T_0,\dots,T_N$. 
We denote by $\sigma = \sum_{i=0}^N s_i T_i$ a general point 
on $H^*(X_\Sigma)$. 
We choose a formal section 
$\frs \colon H^*(X_\Sigma)[\![Q]\!] \to H^*_T(X_\Sigma)[\![Q]\!]$  
of the natural map 
$H^*_T(X_\Sigma)[\![Q]\!] \to H^*(X_\Sigma)[\![Q]\!]$ of the form: 
\[
\frs(\sigma) = \sum_{k\in \bN \cap |\Sigma|} 
\frs_k(\sigma) \phi_k \qquad \text{with} \quad 
\frs_k(\sigma) \in \C[\![Q]\!]
[\![s_0,\dots,s_N]\!]
\] 
such that $\frs_k(0)|_{Q=0} = 0$, $\lim_{|k|\to \infty} \frs_k(\sigma) =0$ 
(in the adic topology of $\C[\![Q]\!][\![s_0,\dots,s_N]\!]$) 
and that the non-equivariant limit of 
$\frs(\sigma)$ equals $\sigma$. 
The section $\frs$ is a morphism 
$\Spf \C[\![Q]\!][\![s_0,\dots,s_N]\!] \to \Spf \C[\![Q]\!][\![\bt]\!]$ 
of formal schemes. 
By pulling back $F$ by $\frs$, we obtain a Landau-Ginzburg potential 
\begin{align*} 
(\frs^* F)(x;\sigma) & = \sum_{k\in \bN \cap |\Sigma|} y_k(\frs(\sigma)) w^{\Psi(k)}
\end{align*} 
parameterized by $\sigma \in H^*(X_\Sigma)$, 
where $y_k = y_k(\tau)$ denotes the inverse mirror map. 
Define the Gauss-Manin system $\GM(\frs^*F)$ of $\frs^* F$ 
to be the cokernel of the map 
\[
zd + d(\frs^*F)\wedge \colon 
\bigoplus_{i=1}^D \C[z]\{\M\}[\![s_0,\dots,s_N]\!] \omega_i  
\to \C[z]\{\M\}[\![s_0,\dots,s_N]\!] \omega. 
\]
\begin{lemma} 
The map $\Theta_\noneq$ in \eqref{eq:Theta_noneq} 
induces an isomorphism: 
\[
\frs^* \Theta_\noneq \colon \GM(\frs^*F) \xrightarrow{\cong} 
H^*(X_\Sigma) [z][\![Q]\!][\![s_0,\dots,s_N]\!]. 
\]
\end{lemma} 
\begin{proof} 
This is slightly subtle as the completed tensor product 
is not right exact in general. 
For any module $M$, we write $M[\![s]\!] = M[\![s_0,\dots,s_N]\!]$ for simplicity. 
The section $\frs$ defines a continuous homomorphism 
$\frs^* \colon \C[z][\![Q]\!][\![\by]\!]\to \C[z][\![Q]\!][\![s]\!]$ 
by $y_k \mapsto y_k(\frs(\sigma))$; this is surjective as $\frs$ is a section. 
We have the commutative diagram: 
\[
\begin{CD}
\bigoplus_{i=1}^D \C[z]\{\M\}[\![\by]\!] \omega_i 
@>{zd + dF\wedge}>> \C[z]\{\M\}[\![\by]\!] \omega 
@>{\Theta_\noneq}>> 
H^*(X_\Sigma) [z][\![Q]\!][\![\by]\!] @>>> 0 
\\ 
@VV{\frs^*}V @VV{\frs^*}V @VV{\frs^*}V \\ 
\bigoplus_{i=1}^D \C[z]\{\M\}[\![s]\!] \omega_i 
@>{zd + d(\frs^*F)\wedge}>> \C[z]\{\M\}[\![s]\!] \omega 
@>{\frs^*\Theta_\noneq}>> H^*(X_\Sigma)[z][\![Q]\!][\![s]\!] 
@>>> 0 
\end{CD} 
\]
where  the top row is exact and the vertical arrows (induced by $\frs$) 
are all surjective. We need to show that the bottom row is exact. 
The surjectivity of $\frs^*\Theta_\noneq$ is obvious. 
Let $\varphi \in \C[z]\{\M\}[\![s]\!]\omega$ be in the kernel 
of $\frs^*\Theta_\noneq$. Choose a lift 
$\hat{\varphi}\in \C[z]\{\M\}[\![\by]\!]\omega$ of $\varphi$ 
such that $\frs^* \hat{\varphi} = \varphi$. 
Then $\frs^* ( \Theta_\noneq(\hat{\varphi})) = 0$. 
When we write $\Theta_\noneq(\hat{\varphi}) 
= \sum_{i=0}^N f_i(z,Q,\by) T_i$ with $f_i \in \C[z][\![Q]\!][\![\by]\!]$, 
this means $\frs^* f_i = 0$ for all $0\le i\le N$. 
Choosing a preimage $\widehat{T}_i \in \C[z]\{\M\}[\![\by]\!] \omega$ of $T_i$ 
under $\Theta_\noneq$, 
we have that $\hat{\varphi} - \sum_{i=0}^N f_i(z,Q,\by) \widehat{T}_i$ 
is in the kernel of $\Theta_\noneq$ and this maps to $\varphi$ under 
$\frs^*$. Now a diagram chasing shows $\varphi$ is in the image of 
$zd + d (\frs^*F)\wedge$. 
\end{proof} 
The pulled-back quantum connection $\frs^* \nabla$ on 
$H^*(X_\Sigma)[z][\![Q]\!][\![s_0,\dots,s_N]\!]$ is given by: 
\begin{align*} 
\frs^* \nabla_{\parfrac{}{s_i}} & = 
\parfrac{}{s_i} + \frac{1}{z} T_i \star_\sigma \\
\frs^*\nabla_{z\parfrac{}{z}} & = \Gr_0 - \frac{1}{z} 
\left(c_1(X_\Sigma) + \sum_{i=0}^N \left(1- \frac{1}{2} \deg T_i\right) 
s_i T_i \right) \star_\sigma 
\end{align*} 
where $\Gr_0 \colon H^*(X_\Sigma)[z] \to H^*(X_\Sigma)[z]$ is 
the linear operator given by $\Gr(\alpha z^n) = ( n + p) \alpha$ for 
$\alpha \in H^{2p}(X_\Sigma)$ and $\star_\sigma$ is the non-equivariant 
quantum product.  
This is the usual quantum connection in non-equivariant theory. 
Thus we obtain: 

\begin{theorem} 
\label{thm:noneq} 
For any formal section $\frs \colon H^*(X_\Sigma)[\![Q]\!] \to 
H^*_T(X_\Sigma) [\![Q]\!]$, the non-equivariant mirror 
isomorphism $\frs^* \Theta_\noneq$ intertwines the quantum connection 
and the Gauss-Manin connection including in the $z$-direction. 
\end{theorem} 

Define the Jacobi ring of $\frs^* F$ to be 
\[
J(\frs^* F) = \C\{\M\}[\![s_0,\dots,s_N]\!]/
\left( x_1 \partial_{x_1} (\frs^*F)(x;\sigma), 
\dots, x_D \partial_{x_D} (\frs^*F)(x;\sigma) \right).  
\]
Exactly in the same way as we deduced Corollary \ref{cor:equiv_Jacobi_ring} 
from Theorem \ref{thm:mirror_isom}, 
we deduce the following corollary from Theorem \ref{thm:noneq}: 

\begin{corollary} 
\label{cor:Jacobiring_noneq} 
We have a $\C[\![Q]\!][\![s_0,\dots,s_N]\!]$-algebra isomorphism 
$J(\frs^*F) \xrightarrow{\cong} 
(H^*(X_\Sigma)[\![Q]\!][\![s_0,\dots,s_N]\!], \star_\sigma)$ 
that sends $\partial_{s_i} (\frs^*F)$ to $T_i$. 
\end{corollary} 

We remark a relationship between $\frs^* F$ and a 
universal unfolding of the original potential $F(x;\ybase)$ 
(see \eqref{eq:y_origin} for $\ybase$). 
We say that $G(x;r)\in \C\{\M\}[\![r_0,\dots,r_N]\!]$ 
is a \emph{universal unfolding} of $F(x;\ybase)$ if 
\begin{itemize} 
\item $G(x;0) = F(x;\ybase)$ and 
\item $\partial_{r_0} G(x;r)|_{r=0},\dots,
\partial_{r_N}G(x;r)|_{r=0}$ form a $\C[\![Q]\!]$-basis of 
the Jacobi ring $J(F(x;\ybase)) = \C\{\M\}/
\langle x_i \partial_{x_i} F(x;\ybase), i=1,\dots, D\rangle$.  
\end{itemize} 
For example, if $\{\phi_i(x)\}_{i=0}^N\subset \C\{\M\}$ 
is a $\C[\![Q]\!]$-basis of the Jacobi ring $J(F(x;\ybase))$, then 
\[
G(x;r) = F(x;\ybase) + \sum_{i=0}^N r_i \phi_i(x) 
\]
gives a universal unfolding of $F(x;\ybase)$. 
We can choose a section $\frs$ so that the map 
$\Spf \C[\![Q]\!][\![s_0,\dots,s_N]\!] \to \Spf \C[\![Q]\!][\![\by]\!]$, 
$\sigma \mapsto \by(\frs(\sigma))=
\{y_k(\frs(\sigma))\}_{k\in \bN\cap |\Sigma|}$ passes through the 
$\C[\![Q]\!]$-valued point $\ybase$, i.e.~
we have $\by(\frs(\sigma^*)) = \ybase$ for some 
$\sigma^*\in H^*(X_\Sigma)[\![Q]\!]$ with 
$\sigma^*|_{Q=0} = 0$. 
Then $\frs^*F|_{\sigma = \sigma^* + \sum_{i=0}^N r_i T_i}$ gives 
a universal unfolding of $F(x;\ybase)$ by Corollary \ref{cor:Jacobiring_noneq} 
(in particular, the Jacobi ring $J(F(x;\ybase))$ is a free $\C[\![Q]\!]$-module 
of rank $\dim H^*(X_\Sigma)$). 
More generally, we have the following. 

% We now compare the present construction 
% of big mirrors with the traditional one 
% \cite{SaitoK:primitiveform,Barannikov:projective, 
% Douai-Sabbah:II}. In the traditional approach, 
% one chooses elements $\{\phi_i(x)\}_{i=0}^N\subset \C\{\M\}$ which 
% form a basis of the Jacobi ring 
% of $F(x;\ybase)$ (see \eqref{eq:y_origin} for $\ybase$) and consider the unfolding 
% $F_{\rm trad}(x;r) := F(x;\ybase) + 
% \sum_{i=0}^N r_i \phi_i(x)$. Here we define the Jacobi ring of 
% $F(x;\ybase)$ to be the quotient of $\C\{\M\}$ by the 
% ideal generated by $x_i \partial_{x_i} F(x;\ybase)$, $i=1,\dots,D$ 
% and $\{[\phi_i(x)]\}_{i=0}^N$ is a basis over $\C[\![Q]\!]$.
\begin{proposition} 
\label{prop:universal_unfolding} 
Let $G(x;r) \in \C\{\M\}[\![r_0,\dots,r_N]\!]$ be an element 
with $G(x;0) = F(x;\ybase)$. The following are equivalent: 
\begin{itemize} 
\item[(1)] $G(x;r)$ is a universal unfolding of $F(x;\ybase)$; 

\item[(2)] there exist a formal section $\frs \colon H^*(X_\Sigma)[\![Q]\!] 
\to H^*_T(X_\Sigma)[\![Q]\!]$ as above 
and a formal invertible change of variables 
$s_j = s_j(r)\in \C[\![Q]\!][\![r_0,\dots,r_N]\!]$ 
with $s_j(0)|_{Q=0}=0$  
such that $G(x;r)=\frs^*F|_{\sigma = \sum_{j=0}^N s_j(r) T_j}$. 
\end{itemize} 
\end{proposition} 
\begin{proof} 
We write $M[\![r]\!] = M[\![r_0,\dots,r_N]\!]$ for simplicity. 
Corollary \ref{cor:Jacobiring_noneq} shows that part (2) implies part (1). 
Conversely, suppose that a universal unfolding 
$G(x;r)$ is given. 
We can write $G(x;r) = F(x;\by(r))$ for some formal morphism 
$r \mapsto \by(r)$, 
$\Spf(\C[\![Q]\!][\![r]\!]) \to \Spf(\C[\![Q]\!][\![\by]\!])$. 
Composing this with the mirror map $\tau(\by)$ and the non-equivariant 
limit, we obtain a map from $\Spf(\C[\![Q]\!][\![r]\!])$ to 
the formal neighbourhood of zero in $H^*(X_\Sigma)[\![Q]\!]$. 
It suffices to show that this is an isomorphism. 
By the formal inverse function theorem, it suffices to 
show that the derivative at $r=0$ is an isomorphism, 
i.e.~the non-equivariant limits of $\parfrac{\tau(\by(r))}{r_i}|_{r=0}$, 
$i=0,\dots,N$ form a basis of $H^*(X_\Sigma)[\![Q]\!]$ 
over $\C[\![Q]\!]$. 
On the other hand, the non-equivariant limit of 
Corollary \ref{cor:equiv_Jacobi_ring} gives 
an isomorphism of $\C[\![Q]\!][\![\by]\!]$-algebras: 
\[
J(F) \cong (H^*(X_\Sigma)[\![Q]\!][\![\bt]\!], \star_\tau) 
\]
which sends $[\partial_{y_k} F(x;\by)]$ to the non-equivariant 
limit of $\parfrac{\tau(\by)}{y_k}$. 
This isomorphism restricted to the base point $\ybase$ 
sends $[\partial_{r_i}G(x;r)|_{r=0}]$ 
to the non-equivariant limit of $\parfrac{\tau(\by(r))}{r_i}|_{r=0}$. 
The conclusion follows. 
\end{proof}

\begin{remark} 
Mirror symmetry for toric varieties has been studied by many people. 
As noted in the Introduction, Givental \cite{Givental:ICM,Givental:toric_mirrorthm} 
and Hori-Vafa \cite{Hori-Vafa} 
proposed Landau-Ginzburg mirrors for toric varieties. 
There are studies on non-compact case (local mirror symmetry) 
\cite{LLY:principleI, Givental:elliptic, CKYZ, 
Konishi-Minabe:localB,Mochizuki_T:twistor_GKZ}, 
Frobenius manifold \cite{Barannikov:projective, 
Douai-Sabbah:II, Iritani:genmir, Reichelt-Sevenheck:logFrob}, 
%non-semipositive case \cite{Iritani:genmir,Iritani:coLef}, 
semi-simplicity \cite{Iritani:coLef,Ostrover-Tyomkin}, 
toric orbifolds \cite{CCLT:wp, Iritani:integral, Guest-Sakai, 
Douai-Mann:wp, Gonzalez-Woodward:tmmp, 
CCIT:mirrorthm, Cheong-CF-Kim,You:Seidel}, an approach using 
Lagrangian Floer theory \cite{FOOO:toricI, 
FOOO:toricII, FOOO:toric_mirrorsymmetry,CLLT:Seidel}, 
tropical geometry \cite{Gross:tropical_P2}, 
quantum Kirwan maps \cite{Woodward:qKirwan, 
Gonzalez-Woodward:tmmp} 
and quasimap spaces \cite{CF-Kim:wallcrossing_genuszero_qmaps, 
Cheong-CF-Kim}, etc. 
In non-semipositive case, we need to take a certain ``$Q$-adic'' completion  
of the Gauss-Manin system; this has been pointed out by the author 
\cite{Iritani:genmir}, \cite[Theorem 1.2]{Iritani:coLef}. 
An isomorphism between a ``completed'' Jacobi ring and 
quantum cohomology was proved by Fukaya-Oh-Ohta-Ono  
\cite[Theorem 1.9]{FOOO:toricI}, \cite[Theorem 1.2.34]{FOOO:toric_mirrorsymmetry} 
using Lagrangian Floer theory 
and by Gonz\'{a}lez-Woodward 
\cite[Theorems 1.16, 4.23]{Gonzalez-Woodward:tmmp} 
using quantum Kirwan map. 
(There are differences on Novikov rings and the choice of 
completions among the literature.)  
In the analytic setting (in semipositive case), 
results analogous to Theorem \ref{thm:noneq} 
are given, e.g.~in \cite[Proposition 4.8]{Iritani:integral}, 
\cite[Theorem 4.11]{Reichelt-Sevenheck:logFrob}, 
\cite[Theorem 5.1.1]{Douai-Mann:wp}, 
\cite[Theorem 7.43]{Mochizuki_T:twistor_GKZ}. 
\end{remark}

\subsection{Reparametrization group} 
In the equivariant setting, the primitive form was given by 
an actual differential form (see \S\ref{subsec:primitive_form}). 
In the non-equivariant setting, 
however, there are many choices for cochain-level primitive 
forms. Consider the commutative diagram: 
\[
\xymatrix{ 
\GM(F_\lambda) \cong \C[z]\{\M\}[\![\by]\!] \omega 
\ar[r]^{\phantom{ABC} \Theta} \ar[d]  & 
H_T^*(X_\Sigma)[z][\![Q]\!][\![\by]\!] 
\ar[d] \\ 
\GM(F) \ar[r]^{\Theta_\noneq \phantom{AB}} 
& H^*(X_\Sigma)[z][\![Q]\!][\![\by]\!]. 
} 
\]
A primitive form can be chosen to be any element in $\C[z]\{\M\}[\![\by]\!]\omega$ 
which maps to $1$ in the bottom right corner. 
We also have the freedom to choose a formal section $\frs \colon H^*(X_\Sigma)[\![Q]\!] 
\to H^*_T(X_\Sigma)[\![Q]\!]$ 
as in \S \ref{subsec:noneq_mirror_isom} 
to pull it back to the base $H^*(X_\Sigma)[\![Q]\!]$. 
We show that all cochain-level primitive forms which are obtained 
in this way and coincide 
with $\omega$ at the origin $\sigma = Q= 0$ are related by 
reparametrizations of the $x$-variables.  

\begin{definition}[reparametrization group] 
Consider a formal change of variables of the form: 
\[
x_i \mapsto 
\tilde{x}_i = x_i \exp
\left(\sum_{k\in \bN \cap |\Sigma|} \epsilon_{k,i} w^{\Psi(k)}\right) 
\qquad 
1\le i\le D, 
\]
where $\bepsilon=\{ \epsilon_{k,i} : k\in \bN \cap |\Sigma|, 1\le i\le D\}$ 
is a set of formal parameters. 
These transformations form a (non-commutative) 
formal group $\cG$ over $\C[\![Q]\!]$ by composition. 
As a formal scheme, $\cG$ is isomorphic to 
$\Spf(\C[\![Q]\!][\![\bepsilon]\!])$.   
We also consider the jet group $J\cG$ of $\cG$, which consists of 
formal transformations: 
\[
x_i \mapsto 
\tilde{x}_i = x_i \exp
\left(\sum_{k\in \bN \cap |\Sigma|} \sum_{n=0}^\infty 
\epsilon_{k,i,n} w^{\Psi(k)} z^n \right)  
\qquad 
1\le i\le D 
\]
containing the parameter $z$. 
The group $J\cG$ is isomorphic to $\Spf(\C[\![Q]\!][\![\tbepsilon]\!])$ 
with $\tbepsilon = \{ \epsilon_{k,i,n} : k \in \bN \cap |\Sigma|, 
1\le i\le D, n=0,1,2,\dots\}$. 
Note that $\cG$ acts on the module $\C\{\M\}$ 
and $J\cG$ acts on the module $\C[z]\{\M\}$. 
Let $e \in J\cG$ denote the identity element. 
The generators $[\partial/\partial \epsilon_{k,i,n}]_{e}$ of the Lie algebra 
$T_e (J\cG)$ correspond to the vector fields: 
\[
W_{k,i,n} :=z^n w^{\Psi(k)} x_i \parfrac{}{x_i} 
\] 
satisfying the commutation relation $[W_{k,i,n}, W_{l,j,p}] 
= Q^{d(k,l)} ( l_i W_{k+l,j,n+p} - k_j W_{k+l,i,n+p})$. 
\end{definition}

The formal group $J\cG$ acts, by change of variables, 
on oscillatory $D$-forms of the form: 
\[
e^{F(x;\by)/z} \omega(\by^+) 
= \exp\left( 
\frac{1}{z}\sum_{k\in \bN \cap |\Sigma|} 
\sum_{n=0}^\infty y_{k,n} z^n w^{\Psi(k)} 
\right)  \omega 
\]
where we set $y_{k,0} = y_k$. 
This defines the $J\cG$-action on $\Spf(\C[\![Q]\!][\![\by,\by^+]\!])$, 
and the generator $W_{k,i,n}$ corresponds to the following vector field:
\[
\tW_{k,i,n} := k_i \parfrac{}{y_{k,n+1}}  + \sum_{l \in \bN\cap |\Sigma|} 
\sum_{j=0}^\infty 
l_i Q^{d(k,l)} y_{l,j} \parfrac{}{y_{k+l,n+j}} 
\]
where the first term in the right-hand side comes from 
the Lie derivative of $\omega = \frac{dx_1}{x_1} 
\cdots \frac{dx_D}{x_D}$. 

\begin{lemma} 
\label{lem:reparametrization} 
Let $\tau(\by)$, $\Upsilon(\by,\by^+,z)$ be the functions 
in Proposition \ref{prop:mirror_map} 
and Theorem \ref{thm:y+}. We have 
\begin{align*} 
\tW_{k,i,n}\tau(\by) & = \lambda_i S_k(\tau(\by)) \delta_{n,0} \\
\tW_{k,i,n}\Upsilon(\by,\by^+,z) & = \lambda_i [z^{n-1} \bS_k(\tau(\by))]_+ 
\Upsilon(\by,\by^+,z). 
\end{align*} 
\end{lemma} 
\begin{proof} 
This is just a calculation. It is obvious that $\tW_{k,j,n} \tau(\by) = 0$ for 
$n>0$. For $n=0$, using the differential equation \eqref{eq:flow}, 
we have 
\begin{align*} 
\tW_{k,i,0} \tau(\by) & = 
\sum_{l\in \bN \cap |\Sigma|} l_i Q^{d(k,l)} y_{l} 
\parfrac{\tau(\by)}{y_{k+l}}
= \sum_{l \in \bN\cap |\Sigma|} l_i Q^{d(k,l)} y_l 
S_{k+l}(\tau(\by)) \\ 
& = \sum_{l\in \bN \cap |\Sigma|} 
l_i y_l S_k(\tau(\by))\star_{\tau(\by)} S_l(\tau(\by)) 
\qquad \text{by Proposition \ref{prop:shift} (3)} \\
& = \lambda_i S_k(\tau(\by)) 
\qquad \qquad \qquad \qquad \qquad 
\quad \; \, \text{by Remark \ref{rem:linear_relation}}. 
\end{align*} 
On the other hand, using the differential equation 
\eqref{eq:Upsilon_y+flow}, we have 
{\allowdisplaybreaks
\begin{align*} 
&\tW_{k,i,n} \Upsilon  = k_i z^n \bS_k(\tau(\by)) \Upsilon
+ \sum_{l\in \bN\cap |\Sigma|}\sum_{j=0}^\infty  
l_i Q^{d(k,l)} y_{l,j} 
[z^{n+j-1}\bS_{k+l}(\tau(\by))]_+ \Upsilon \\ 
& = k_i z^n \bS_k(\tau(\by)) \Upsilon 
+ \sum_{l\in \bN \cap |\Sigma|} 
\sum_{j=0}^\infty l_i y_{l,j} z^{n+j-1} 
\bS_{k}(\tau(\by)) \bS_{l}(\tau(\by)) \Upsilon \\
& \quad - \delta_{n,0} \sum_{l\in \bN\cap |\Sigma|} 
l_i y_l  z^{-1} S_k(\tau(\by)) \star_{\tau(\by)} 
S_l(\tau(\by)) \star_{\tau(\by)} \Upsilon
\qquad \text{by Proposition \ref{prop:shift} (3)} \\ 
& = z^n \bS_k(\tau(\by)) \left[ 
k_i  \Upsilon 
+ \sum_{l\in \bN \cap |\Sigma|} 
\sum_{j=0}^\infty l_i y_{l,j} 
\left( \parfrac{\Upsilon}{y_{l,j}} + \delta_{j,0} z^{-1} 
S_l(\tau(\by)) \star_{\tau(\by)} \Upsilon\right)  
\right] \\ 
& \quad - \delta_{n,0} z^{-1} \lambda_i 
S_k(\tau(\by)) \star_{\tau(\by)} \Upsilon
\qquad \qquad \qquad \qquad \qquad \quad \;\;
\text{by Remark \ref{rem:linear_relation}} \\ 
& = z^n \bS_k(\tau(\by)) (k_i + z^{-1} \lambda_i) \Upsilon 
- \delta_{n,0} \lambda_i S_k(\tau(\by)) \star_{\tau(\by)} \Upsilon 
= \lambda_i [z^{n-1} \bS_k(\tau(\by))]_+ \Upsilon. 
\end{align*} 
where in the last line we used the equation $\sum_{l\in \bN\cap |\Sigma|} 
\sum_{j=0}^\infty l_i y_{l,j} \partial_{y_{l,j}} \Upsilon =0$, 
which follows from the expansion \eqref{eq:expansion_Upsilon_y+}. 
}
\end{proof} 

Let $R$ be a linearly topologized $\C[\![Q]\!]$-algebra. 
Let $R_\nilp$ denote the ideal of $R$ consisting of topologically 
nilpotent elements, i.e.~elements $x\in R$ such that $\lim_{n\to \infty} 
x^n =0$.  We assume that $R$ is complete and Hausdorff and that, 
for any neighbourhood $U$ of zero in $R$, there exists $n\in \N$ such that 
$(R_\nilp)^n \subset U$.  
An $R$-valued point on $\Spf (\C[\![Q]\!][\![\by,\by^+]\!])$ 
is given by a collection $\{\log y_1,\dots,\log y_m\} 
\cup \{y_{k}: k\in G\} \cup \{ y_{n,k} : k\in \bN \cap |\Sigma|, n=1,2,\dots \}$ 
of elements of $R_\nilp$ such that every neighbourhood $U\subset R$ of $0$ 
contains all but finitely many elements of this collection. 
An $R$-valued point on the formal group $J\cG$ is described similarly. 
For an $R$-valued point $(\by,\by^+)$, we can make sense of 
$\tau(\by) \in H^*_T(X_\Sigma)\hotimes_{\C} R$ and 
$\Upsilon(\by,\by^+,z) \in H^*_T(X_\Sigma)[z]\hotimes_{\C} R$ 
where $\hotimes$ denotes the completed tensor product. 
Consider the map  
\[
(\by,\by^+) \mapsto (\sigma(\by), \Xi(\by,\by^+,z)) 
\in H^*(X_\Sigma)[\![Q]\!] \times H^*(X_\Sigma)[z][\![Q]\!]
\]   
defined as the non-equivariant limit of 
$(\tau(\by), \Upsilon(\by,\by^+,z))$. 
We will see that this map classifies the $J\cG$-orbit 
on the space $\Spf (\C[\![Q]\!][\![\by,\by^+]\!])$ 
of oscillatory forms. 
\begin{theorem} 
\label{thm:orbit_space}
Let $R$ be as above, and 
let $(\by_1,\by^+_1)$, $(\by_2,\by^+_2)$ be two $R$-valued points 
of $\Spf \C[\![Q]\!][\![\by,\by^+]\!]$. 
Then the following are equivalent:  
\begin{enumerate} 
\item there exists an $R$-valued point $g\in J\cG(R)$ of $J\cG$ 
such that $g \cdot [e^{F(x;\by_1)/z} \omega(\by_1^+)] 
= [e^{F(x;\by_2)/z} \omega(\by_2^+)]$; 
\item one has $\sigma(\by_1) = \sigma(\by_2)$ and 
$\Xi(\by_1,\by_1^+,z) = \Xi(\by_2,\by_2^+,z)$. 
\end{enumerate} 
\end{theorem} 
\begin{proof} 
(1) $\Rightarrow$ (2): Since the exponential map identifies 
the formal neighbourhood of the origin of $T_e(J\cG)$ with $J\cG$, 
we can work at the Lie algebra level. 
Lemma \ref{lem:reparametrization} implies that the generators 
$\tW_{k,i,n}$ of $T_e(J\cG)$ act on $\sigma(\by)$, $\Xi(\by,\by^+,z)$ 
trivially. Thus $\sigma$ and $\Xi$ are constant along a $J\cG$-orbit. 

(2) $\Rightarrow$ (1): 
It follows from \eqref{eq:tau_Jacobi}, \eqref{eq:Upsilon_Jacobi} that 
\[
\left. 
\left(\delta_{n,0} S_k(\tau(\by)), 
[z^{n-1}\bS_k(\tau(\by))]_+\Upsilon(\by,\by^+,z) \right)
\right|_{\by= \ybase,\by^+=0, Q=0} 
\]
form a $\C$-basis of $H^*_T(X_\Sigma)\times H^*_T(X_\Sigma)[z]$. 
Define a $\C$-basis $\{ v_{l,p} : l\in \bN \cap |\Sigma|, p=0,1,2,\dots\}$ 
of $H^*_T(X_\Sigma) \times H^*_T(X_\Sigma)[z]$ by 
\[
v_{l,p} =
\begin{cases} 
(\phi_l,0) & p = 0; \\ 
(0, z^{p-1} \phi_l) & p>0.  
\end{cases} 
\] 
Then we can write 
\begin{align*} 
v_{l,p} = 
\sum_{k\in \bN \cap |\Sigma|} 
\sum_{n=0}^\infty 
c_{k,n,l,p} 
\left(\delta_{n,0} S_k(\tau(\by)), [z^{n-1}\bS_k(\tau(\by))]_+ 
\Upsilon(\by,\by^+,z) \right) 
\end{align*} 
for some (unique) coefficients $c_{k,n,l,p} 
\in \C[\![Q]\!][\![\by,\by^+]\!]$ 
such that $\lim_{|k|+n\to \infty} c_{k,n,l,p} =0$. 
Define the vector fields $\frX_{l,p,i} = 
\sum_{k\in \bN \cap |\Sigma|} \sum_{n=0}^\infty 
c_{k,n,l,p} \tW_{k,i,n}$. 
Then Lemma \ref{lem:reparametrization} implies: 
\[
\frX_{l,p,i} \left(\tau(\by), \Upsilon(\by,\by^+,z)\right) 
= \lambda_i v_{l,p}. 
\]
When we assume (2), we can find $r_{l,p,i} \in R_\nilp$ such that 
$\lim_{|l|+p \to \infty} r_{l,p,i} =0$ and  
\[
\left(\tau(\by_2), \Upsilon(\by_2,\by_2^+,z)\right) - 
\left(\tau(\by_1),\Upsilon(\by_1,\by_1^+,z)\right) 
= \sum_{l\in \bN \cap |\Sigma|} \sum_{p=0}^\infty 
\sum_{i=1}^D r_{l,p,i} \lambda_i v_{l,p}. 
\]
Such $r_{l,p,i}$ are not unique. They 
define an ``$R$-dependent'' vector field $\sum_{l,p,i} r_{l,p,i} \frX_{l,p,i}$  
on the formal scheme $\cM_R :=\Spf(\C[\![Q]\!][\![\by,\by^+]\!]) 
\times_{\Spf(\C[\![Q]\!])} \Spf(R) = \Spf(R[\![\by,\by^+]\!])$ over $R$. 
Since $r_{l,p,i}$ are topologically nilpotent, we can integrate 
this vector field to obtain an automorphism 
of $\cM_R$ over $R$ (see Theorem \ref{thm:formal_flow}). 
By construction, this automorphism 
sends the $R$-valued point $(\by_1,\by_1^+)$ to 
$(\by_2,\by_2^+)$. By integrating the corresponding Lie algebra 
element in the formal group $(J\cG)_R = J\cG \times_{\Spf \C[\![Q]\!]} \Spf(R)$ 
over $R$, 
we obtain an element $g\in J\cG(R)$ which sends 
$e^{F(x;\by_1)/z} \omega(\by_1^+)$ to 
$e^{F(x;\by_2)/z} \omega(\by_2^+)$.  
\end{proof} 

To obtain a primitive form in the non-equivariant setting, 
we need to choose a formal section 
$\frs \colon H^*(X_\Sigma)[\![Q]\!]  
\to H^*_T(X_\Sigma)[\![Q]\!]$ 
as in \S\ref{subsec:noneq_mirror_isom} 
\emph{and} a formal function $\frf \colon H^*(X_\Sigma)[\![Q]\!] 
\to H^*_\hT(X_\Sigma)[\![Q]\!]$ 
that is expanded as 
\[
\frf(\sigma) = \sum_{k\in \bN \cap |\Sigma|} \sum_{n=0}^\infty 
\frf_{k,n}(\sigma) z^n \phi_k 
\]
with $\frf_{k,n}(\sigma)\in \C[\![Q]\!][\![s_0,\dots,s_N]\!]$ 
such that $\frf_{k,n}(0)|_{Q=0} = 
\delta_{k,0} \delta_{n,0}$, $\lim_{|k| + n \to \infty} 
\frf_{k,n}(\sigma) = 0$ in the topology of $\C[\![Q]\!][\![s_0,\dots,s_N]\!]$ 
and that the image of $\frf(\sigma)$ under the natural map 
$H^*_\hT(X_\Sigma)[\![Q]\!] \to H^*(X_\Sigma)[z][\![Q]\!]$ is $1$. 
By the isomorphism in Lemma \ref{lem:very_big_mirrormap}, 
we obtain a Landau-Ginzburg potential $\frs^* F=
F(x;\by(\sigma))$ and 
a primitive form $\zeta_{(\frs,\frf)}=\omega(\by^+(\sigma))$ 
such that 
\[
\frs(\sigma)= \tau(\by(\sigma)), 
\qquad 
\frf(\sigma) = \Upsilon(\by(\sigma),\by^+(\sigma),z). 
\] 
The cohomology class $[\zeta_{(\frs,\frf)}]$ maps to $1$ under 
the isomorphism $\frs^* \Theta_{\rm noneq} \colon \GM(\frs^*F) 
\cong H^*(X_\Sigma)[\![Q]\!][\![s_0,\dots,s_N]\!]$. 
\begin{corollary} 
\label{cor:noneq_primitive}
Any oscillatory primitive forms $\exp(\frs^*F/z)\zeta_{(\frs,\frf)}$ associated to 
various data $(\frs,\frf)$ as above are related to each other by a co-ordinate 
change in the $x$-variables, i.e.~they are contained in a single 
$J\cG(\C[\![Q]\!][\![s_0,\dots,s_N]\!])$-orbit. 
\end{corollary}

\section{Extended $I$-function}
\label{sec:I-function} 

In this section we relate the mirror map $\tau(\by)$, the function 
$\Upsilon(\by,z)$ (or $\Upsilon(\by,\by^+,z)$) and the primitive form 
$\zeta$ 
with certain hypergeometric series called the (extended) $I$-function. 
This gives us a concrete algorithm to calculate these quantities, 
although actual computations could be very complicated. 

\begin{definition}[\cite{CCIT:mirrorthm}]
Define a cohomology-valued hypergeometric series in the 
variables $\by = \{y_k : k\in \bN \cap |\Sigma|\}$ as follows: 
\[
I(\by,z) = z e^{\sum_{i=1}^m u_i \log y_i/z} 
\sum_{\ell \in \hLL_\eff} \by^\ell Q^{d(\ell)}
\left( \prod_{i=1}^m 
\frac{\prod_{c=-\infty}^0 (u_i+ cz)}
{\prod_{c=-\infty}^{\ell_{b_i}}(u_i+cz)} \right) 
\frac{1}{\prod_{k \in G} \ell_k! z^{\ell_k}} 
\]
where we used the notation from \S \ref{subsec:mirrormap}. 
This belongs to $H_\hT^*(X_\Sigma)_\loc[\![Q]\!][\![\by]\!]$ 
and is called \emph{the extended $I$-function}. 
\end{definition} 

Recall from Remark \ref{rem:Givental_cone} that the image 
of the fundamental solution $M(\tau,z)$ sweeps the 
Givental cone in $H^*_\hT(X_\Sigma)_\loc$. 
We show that the extended $I$-function is on the 
Givental cone \cite{Givental:toric_mirrorthm, CCIT:mirrorthm}. 

\begin{proposition} 
\label{prop:ext_I}
Let $\tau(\by)$, $\Upsilon(\by,z)$ denote the functions 
from Proposition \ref{prop:mirror_map}.  
We have $I(\by,z) = z M(\tau(\by),z) \Upsilon(\by,z)$. 
\end{proposition} 
\begin{proof} 
This proposition was proved in \cite[\S 4.3]{Iritani:shift} 
along the locus $\{y_k=0 : k\in G\}$. It suffices to show 
that both $I(\by,z)$ and $zM(\tau(\by),z)\Upsilon(\by,z)$ satisfy 
the same differential equation in $y_k$ for $k\in G$. We claim 
that both functions satisfy: 
\[
\parfrac{\bbf(\by,z)}{y_k} = z^{-1} \cS_k \bbf(\by,z).  
\]
Since $\bV_k$ corresponds to the linear vector field 
$\bbf \mapsto z^{-1} \cS_k \bbf$ on the Givental space $H_\hT(X_\Sigma)_\loc$ 
(see Remark \ref{rem:Givental_cone}, \cite[\S 4.3]{Iritani:shift}), 
the differential equation holds for $\bbf = z M(\tau(\by),z) \Upsilon(\by,z)$. 
We show that the differential equation holds for $\bbf = I(\by,z)$. 
Let $x\in X_\Sigma$ be a $T$-fixed point. 
Let $I_x(\by,z)$ denote the restriction of $I(\by,z)$ to $x$. 
By Definition \ref{def:shift_Givental}, we need to show that: 
\begin{equation}
\label{eq:diffeq_I_x} 
z \parfrac{}{y_k} I_x(\by,z) = \Delta_x(k) e^{-k z \partial_\lambda} 
I_x(\by,z)  
\end{equation} 
with 
\[
\Delta_x(k) = Q^{\sigma_x - \sigma_{\min}(k)} 
\prod_{i=1}^m \frac{\prod_{c=-\infty}^0 (u_i(x) +cz)}
{\prod_{c=-\infty}^{-u_i(x) \cdot k} (u_i(x) + c z)} 
\]
where $u_i(x) \in H^2_T(\pt)$ denotes the restriction of 
$u_i$ to $x$. Recall that $\sigma_{\min}(k)\in H_2^\sec(E_k)$ 
corresponds to $-\Psi(k) \in \Z^m \cong H_2^T(X_\Sigma,\Z)$ 
by Lemma \ref{lem:minimal_section}. A similar argument 
shows that the section class $\sigma_x$ associated to the fixed point $x$ 
corresponds to $(-u_i(x)\cdot k)_{i=1}^m$.   
Therefore the right-hand side of \eqref{eq:diffeq_I_x} equals: 
\begin{multline*} 
Q^{\Psi(k) - \sum_{i=1}^m (u_i(x) \cdot k)e_i}  
\prod_{i=1}^m \frac{\prod_{c=-\infty}^0 (u_i(x) +cz)}
{\prod_{c=-\infty}^{-u_i(x) \cdot k} (u_i(x) + c z)} 
z e^{\sum_{i=1}^m (u_i(x) \log y_i/z - (u_i(x) \cdot k) \log y_i)} \\ 
\times 
\sum_{\ell\in \hLL_\eff} Q^{d(\ell)} \by^\ell 
\left( \prod_{i=1}^m 
\frac{\prod_{c=-\infty}^{-u_i(x) \cdot k} (u_i(x)+ cz)}
{\prod_{c=-\infty}^{\ell_{b_i}-u_i(x) \cdot k}(u_i(x)+cz)} \right) 
\frac{1}{\prod_{l\in G} \ell_l! z^{\ell_l}} 
\end{multline*} 
Note that we can write the extended $I$-function as a sum 
over $\ell \in \hLL$ (by replacing $\ell_l!$ with $\Gamma(1+\ell_l)$) 
since the summand corresponding 
to $\ell \notin \hLL_\eff$ automatically vanishes. 
Shifting the index $\ell$ as $\ell \to \ell + \sum_{i=1}^m 
(u_i(x) \cdot k) e_{b_i} - e_k$, we find that this equals 
the left-hand side of \eqref{eq:diffeq_I_x}. 
% \[
% z e^{\sum_{i=1}^m u_i \log y_i/z} \sum_{\ell \in \hLL} 
% \frac{\ell_k z}{y_k} Q^{d(\ell)} \by^\ell 
% \left( \prod_{i=1}^m 
% \frac{\prod_{c=-\infty}^{0} (u_i(x)+ cz)}
% {\prod_{c=-\infty}^{\ell_{b_i}}(u_i(x)+cz)} \right) 
% \frac{1}{\prod_{l\in G} \ell_l! z^{\ell_l}} 
% \]
\end{proof} 

We explain that the functions $(\tau(\by), \Upsilon(\by,z))$ are obtained 
from $I(\by,z)$ via the Birkhoff factorization \cite{Coates-Givental}. 
Consider the $\C[z][\![Q]\!][\![\by]\!]$-linear map 
\[
dI \colon 
H_T^*(X_\Sigma)[z] [\![Q]\!][\![\by]\!] \to 
H_T^*(X_\Sigma)(\!(z^{-1})\!)[\![Q]\!][\![\by]\!]
\]
sending $\phi_k$ to $\parfrac{I(\by,z)}{y_k}$ for 
$k\in \bN \cap |\Sigma|$. 
Here we used the embedding $H_\hT^*(X_\Sigma)_\loc \hookrightarrow 
H_T^*(X_\Sigma)(\!(z^{-1})\!)$ given by the Laurent expansion 
at $z=\infty$. 
We have 
\begin{align*} 
\parfrac{I(\by,z)}{y_k} & = M(\tau(\by),z) \left( 
z \parfrac{\Upsilon(\by,z)}{y_k} + 
\parfrac{\tau(\by)}{y_k} \star_{\tau(\by)} \Upsilon(\by,z) 
\right) \\ 
& = M(\tau(\by),z) \bS_k(\tau(\by))  \Upsilon(\by,z). 
\end{align*} 
Let $\bS\Upsilon 
\colon H_T^*(X_\Sigma)[z][\![Q]\!][\![\by]\!] 
\to H_T^*(X_\Sigma)[z][\![Q]\!][\![\by]\!]$ 
denote 
the $\C[z][\![Q]\!][\![\by]\!]$-linear map 
sending $\phi_k$ to $\bS_k(\tau(\by)) \Upsilon(\by,z)$ 
for each $k\in \bN \cap |\Sigma|$. 
Then we have: 
\[
d I = M(\tau(\by),z) \circ \bS \Upsilon. 
\]
This can be viewed as the Birkhoff factorization of $dI$ 
when we regard $z$ as a loop parameter; 
notice that $M(\tau(\by),z)$ belongs to $\End_\C(H_T^*(X_\Sigma))[\![z^{-1}]\!]
[\![Q]\!][\![\by]\!]$ and that $M(\tau(\by),z=\infty) = \id$. 
The Birkhoff factorization can be performed recursively 
in powers in $Q$ and $\by$, and this gives a concrete 
algorithm to compute $M(\tau(\by),z)$ and $\Upsilon(\by)$. 
The mirror map $\tau(\by)$ is then obtained from the expansion: 
\begin{equation} 
\label{eq:J_asymptotics} 
M(\tau(\by),z) 1 = 1 + \frac{\tau(\by)}{z} + o(z^{-1}). 
\end{equation} 
Once we obtain $\tau(\by)$ and $\bS\Upsilon$, 
we can calculate the inverse mirror map $\by = \by(\bt)$ and 
the primitive form $\zeta = \sum_{k\in \bN\cap|\Sigma|} 
c_k(\by,z) w^{\Psi(k)} \omega$ by the requirement (see \S \ref{subsec:primitive_form})
\[
\sum_{k\in \bN\cap|\Sigma|} 
c_k(\by,z) \bS_k(\tau(\by)) \Upsilon(\by,z) = 1. 
\]

Finally we extend Proposition \ref{prop:ext_I} to the function $\Upsilon(\by,\by^+,z)$ 
in Theorem \ref{thm:y+} and describe an alternative method to 
calculate the primitive form. 
Let $\by^+ = \{y_{k,n} : k\in \bN \cap |\Sigma|, n=1,2,3,\dots\}$ be the 
variables in \S \ref{subsec:primitive_form} and consider 
\[
y_k(z) = y_k + \sum_{n=1}^\infty y_{k,n} z^n. 
\]
We write $\by(z) = \{ y_k(z) : k\in \bN\cap |\Sigma|\}$. 
\begin{proposition} 
\label{prop:verybig_I}
Let $\Upsilon(\by,\by^+,z)$ be as in Theorem \ref{thm:y+}. 
We have 
\[
I(\by(z),z) = z M(\tau(\by),z) \Upsilon(\by,\by^+,z) 
\]
where $I(\by(z),z)$ is obtained from the extended $I$-function 
$I(\by,z)$ by replacing $y_k$ with $y_k(z)$. 
\end{proposition}
\begin{proof} 
It suffices to show that both sides satisfy the same differential 
equation: 
\[
\parfrac{\bbf(\by,\by^+)}{y_{k,n}} = z^{n-1} \cS_k \bbf(\by,\by^+). 
\]
In the proof of Proposition \ref{prop:ext_I}, we showed that 
$z \parfrac{I(\by,z)}{y_k} = \cS_k I(\by,z)$. Thus 
$\bbf = I(\by(z),z)$ satisfies the above differential equation. 
The differential equation for $\bbf = z M(\tau(\by),z) \Upsilon(\by,\by^+,z)$
follows easily from Proposition \ref{prop:shift} (1) and Theorem \ref{thm:y+}. 
\end{proof} 

Recall from \S \ref{subsec:primitive_form} 
that the primitive form $\zeta$ is given by 
$\omega(\by^+)$ in \eqref{eq:omega_y+} such 
that $\Upsilon(\by,\by^+,z) = 1$. 
Thus the asymptotics \eqref{eq:J_asymptotics} implies: 

\begin{corollary} 
\label{cor:primitive} 
The primitive form $\zeta$ is given by $\omega(\by^+)$ 
for $\by^+$ such that the asymptotics $I(\by(z),z) = z(1+ O(z^{-1}))$ holds.  
Moreover, for such $\by^+$, the asymptotics $I(\by(z),z) = z + \tau(\by) + O(z^{-1})$ 
determines the mirror map $\tau(\by)$. 
\end{corollary} 

\begin{remark} 
Proposition \ref{prop:verybig_I} implies that $I(\by(z),z)$ lies on the Givental cone. 
In fact, the family of vectors $(\by,\by^+) \mapsto I(\by(z),z)$ covers the whole 
Givental cone and $(\by,\by^+)$ may be viewed as a B-model co-ordinate 
system on the cone. 
\end{remark} 

\begin{remark} 
\label{rem:cone_quotient}
When we identify the space $\Spf(\C[\![Q]\!][\![\by,\by^+]\!])$ with 
the Givental cone as in the above remark, we can interpret 
Theorem \ref{thm:orbit_space} as follows: \emph{the non-equivariant 
Givental cone is the orbit space 
of the equivariant Givental cone under the action of the group $J\cG$ 
of reparametrizations of the mirror}. 
\end{remark} 

\begin{remark} 
The formal geometry appearing in this paper is very similar to 
the treatment of the Givental cone as a formal scheme 
in \cite[Appendix B]{CCIT:computing}. 
\end{remark} 

\begin{remark} 
It should be possible to generalize 
the results in this paper to toric orbifolds (or toric Deligne-Mumford stacks). 
This is interesting  
since toric orbifolds correspond to arbitrary Laurent polynomials. 
See \cite{CCLT:wp,Guest-Sakai, Douai-Mann:wp, Gonzalez-Woodward:tmmp, 
CCIT:mirrorthm, Cheong-CF-Kim, You:Seidel} for related works.  
\end{remark}

\appendix
\section{Formal geometry in infinite dimensions}
For the sake of completeness, we prove a formal inverse function theorem and 
the existence of a flow of a vector field in infinite dimensions. 
The results here are straightforward generalizations 
of well-known results in finite dimensions, 
but we could not find a reference. 
Throughout the section, we assume that $R$ is a linearly topologized ring 
containing $\Q$ and that $R$ is complete and Hausdorff. 
We denote by $\{R_\nu\}$ a fundamental neighbourhood system 
of zero consisting of ideals of $R$. 

Let $\bx = \{x_1,x_2,x_3,\dots\}$ be a countably infinite 
set of variables. 
A morphism $f \colon \Spf(R[\![\bx]\!]) \to \Spf(R[\![\bx]\!])$ of 
formal schemes over $R$ (see \S \ref{subsec:power_series} 
for $R[\![\bx]\!]$) is given by a tuple 
$\{f^*(x_1),f^*(x_2), f^*(x_3),\dots\}$ 
of elements in $R[\![\bx]\!]$ such that $f^*(x_i)|_{\bx =0}\in R_\nilp$ 
and $\lim_{n\to \infty} f^*(x_n) =0$, 
where $R_\nilp = \{x \in R : \lim_{n\to \infty} x^n = 0\}$. 
Consider the $R$-module 
\[
T := \left\{(r_n)_{n=1}^\infty \in R^{\N} : 
\lim_{n\to\infty} r_n = 0 \right\} 
\cong (\Q^{\oplus \N})\hotimes R. 
\] 
The topology on $T$ is defined by submodules 
$(\Q^{\oplus \N})\hotimes R_\nu$.  
A morphism $f$ associates the (continuous) 
tangent map $df \colon T \to T$ defined by $df (e_i) = \sum_{j=1}^\infty 
\left.\parfrac{f^*(x_j)}{x_i}\right|_{\bx=0} e_j$. 
The following gives two important classes of morphisms.  
\begin{itemize} 
\item for a continuous $R$-module homomorphism  
$A \colon T \to T$ with $A(e_i) = \sum_{j=1}^\infty a_{ji} e_j$,  
we have a linear map $f$ given 
by $f^*(x_j) = \sum_{j=1}^\infty a_{ji} x_i$; 

\item for an element $(r_j)_{j=1}^\infty \in T$ with $r_j \in R_\nilp$, 
we have a translation map $f$ given by $f^*(x_j) = x_j + r_j$. 
\end{itemize} 

\begin{theorem}[formal inverse function theorem] 
\label{thm:formal_inverse_function} 
Let $f \colon \Spf (R[\![\bx]\!]) \to \Spf(R[\![\bx]\!])$ be a morphism of 
formal schemes over $R$. If the tangent map $df \colon T \to T$ at 
$\bx=0$ is an isomorphism, $f$ is an isomorphism.  
\end{theorem} 
\begin{proof} 
By composing with a linear map and a translation, 
we may assume that $f(0) = 0$ and the tangent map $df$ is the identity. 
Then the truncation of $f^*$ given by 
$R[\![x_1,\dots,x_n]\!] \subset R[\![\bx]\!] \xrightarrow{f^*} 
R[\![\bx]\!] \twoheadrightarrow R[\![x_1,\dots,x_n]\!]$ 
is an isomorphism, by the inverse function theorem in finite 
dimensions (see \cite[Appendix A]{Hazewinkel}; the proof over a discrete 
ring works verbatim over $R$). 
It follows easily that $f^*$ is an isomorphism. 
\end{proof} 

Next we discuss the integrability of a formal vector field. 
A \emph{formal vector field} on $\Spf(R[\![\bx]\!])$ over $R$ is 
a formal sum $V= \sum_{n=1}^\infty V_n(\bx) \parfrac{}{x_n}$ 
with $V_n(\bx) \in R[\![\bx]\!]$ such that 
$\lim_{n \to \infty} V_n(\bx) = 0$. 
We consider the flow $t\mapsto \bx(t)=(x_n(t))_{n=1}^\infty$ satisfying 
\begin{equation}
\label{eq:formal_flow}
\frac{dx_n(t)}{dt} = V_n(\bx(t)) \qquad 
\text{with $x_n(0) = x_n$}. 
\end{equation} 
\begin{theorem} 
\label{thm:formal_flow} 
There exists a unique solution $\bx(t) = (x_n(t))_{n=1}^\infty$ to the 
equation \eqref{eq:formal_flow} which defines a morphism 
$\Spf(R[\![\bx]\!][\![t]\!]) \to \Spf(R[\![\bx]\!])$ of formal schemes. 
Let $I\subset R$ be an ideal such that, 
for any $\nu$, there exists $n\in \N$ such that $I^n \subset R_\nu$. 
If $V_n(\bx)\in I[\![\bx]\!]$ for all $n$, then the  
substitution $t=1$ in the solution $\bx(t)$ is well-defined and we 
obtain a time-one flow map $\Spf(R[\![\bx]\!]) \to \Spf(R[\![\bx]\!])$. 
\end{theorem} 
\begin{proof} 
Note that $V$ defines a well-defined continuous mapping 
$V \colon R[\![\bx]\!] \to R[\![\bx]\!]$. 
The flow is given by a continuous 
ring homomorphism $R[\![\bx]\!] \to R[\![\bx]\!][\![t]\!]$ 
defined by $\varphi \mapsto \exp(t V) \varphi = \sum_{k=0}^\infty 
\frac{1}{k!} t^k V^k(\varphi)$, where $V^k$ is the $k$-fold composition of $V$, 
see \cite[3C]{Ilyashenko-Yakovenko}. 
The former statement follows. To see the latter, it suffices to 
notice that $\lim_{k\to \infty} V^k(\varphi) = 0$ 
uniformly for all $\varphi \in R[\![\bx]\!]$ under the assumption. 
\end{proof}

\bibliographystyle{plain}
\bibliography{shift_mirror}
\end{document}